\documentclass[reqno,twoside,12pt,a4paper]{amsart}
\usepackage{layout}
\usepackage{latexsym}
\usepackage{amsmath}
\usepackage{amssymb}
\usepackage{ascmac}
\usepackage{mathrsfs}
\usepackage{cases} 
\usepackage{caption}
\usepackage{graphicx}
\usepackage{cite}
\usepackage{color}
\definecolor{viola}{rgb}{0.3,0,0.7}
\definecolor{ciclamino}{rgb}{0.5,0,0.5}
\definecolor{blu}{rgb}{0,0,0.7}
\definecolor{rosso}{rgb}{0.85,0,0}

\renewcommand{\d}{\, {\mathrm d}}
\newcommand{\dx}{\, {\mathrm d} x}
\newcommand{\dg}{\, {\mathrm d} \Gamma}

\newcommand{\ds}{\, {\mathrm d} s}
\newcommand{\dt}{\, {\mathrm d} t}
\newcommand{\dtau}{\, {\mathrm d} \tau}

\numberwithin{equation}{section}

\newtheorem{theorem}{Theorem}[section]

\newtheorem{lemma}[theorem]{Lemma}

\newtheorem{remark}[theorem]{Remark}

\renewenvironment{proof}{\noindent {\bf Proof.}}{\hfill $\Box$}

\allowdisplaybreaks[4]

\title[Permeability parameter asymptotics in a Cahn--Hilliard system]
{Permeability parameter asymptotics in a Cahn--Hilliard system with third type transmission conditions}

\author[P.\ Colli]{Pierluigi Colli}
\address{Pierluigi Colli: Dipartimento di Matematica ``F. Casorati'', Universit\`a di Pavia
and Research Associate at the IMATI -- C.N.R. Pavia, via Ferrata 5, 27100 Pavia, Italy}
\email{pierluigi.colli@unipv.it}

\author[T.\ Fukao]{Takeshi Fukao}
\address{Takeshi Fukao: Faculty of Advanced Science and Technology, Ryukoku University, 
1-5 Yokotani, Seta Oe-cho, Otsu-shi, Shiga 520-2194, Japan}
\email{fukao@math.ryukoku.ac.jp}

\author[K.\ F.\ Lam]{Kei Fong Lam}
\address{Kei Fong Lam: Department of Mathematics, 
Hong Kong Baptist University, 
Kowloon Tong, Hong Kong}
\email{akflam@hkbu.edu.hk}

\dedicatory{}
\begin{document}

\thispagestyle{empty}

\begin{abstract}

In this paper, 
we study a system in which the Cahn--Hilliard system is imposed in the bulk domain 
and an Allen--Cahn type equation is prescribed on the boundary, connected through a third type transmission condition characterized by a permeability parameter. 
This setting is closely related to the theory of transmission problems, and the third type transmission condition can be viewed as a remnant of a thin-boundary description, with the permeability parameter controlling the degree of coupling between the bulk and the boundary. 
The main objective is to perform a rigorous asymptotic analysis with respect to the permeability parameter. We investigate two limits: the parameter tending to zero, corresponding to a completely impermeable boundary, and the parameter tending to infinity, corresponding to a perfectly permeable boundary where the bulk and boundary phases are fully coupled. 
For each limiting regime, we establish the convergence of solutions to the respective limit problems and characterize the resulting equations. 
\smallskip

\end{abstract}
\maketitle

\section{Introduction}
\label{intro}
\setcounter{equation}{0} 
Dynamic boundary conditions are closely connected to transmission problems.
To solve partial differential equations in domains with boundaries, 
dynamic boundary conditions involving time derivatives have been actively studied in recent years, 
along with Dirichlet, Neumann, and third\footnote{It was communicated to the authors by J.F.~Rodrigues that the commonly-termed third type/Robin boundary condition $K\partial_{\boldsymbol{\nu}} u = \phi - u$ should have been attributed to J.\ Fourier rather than to V.G.\ Robin, as there was no record of the boundary condition appearing in the latter's work (see, e.g., K.\ Gustafson, and T.\ Abe, 
The third boundary condition --- was it Robin's?, 
Math.\ Intelligencer, {\bf 20}, (1998) 63--71). In support of this observation, the authors deliberately adopt the term ``third type'' (boundary condition, or transmission condition) throughout this paper.} type boundary conditions. 
On the other hand, a transmission problem is a type of boundary value problem where two conditions are prescribed on the common interface. While this appears to be overdetermined, the conditions contain an additional unknown variable, thereby maintaining the well-posedness of the problem. 
From a modeling perspective, boundary conditions are auxiliary equations that supplement the main equation. However, when the boundary possesses a finite thickness, one must account for the dynamics not only in the main domain but also within the boundary region. 
From this viewpoint, recent research has focused on initial-boundary value problems for various partial differential equations under dynamic boundary conditions, together with related transmission problems and their asymptotic analysis. Meanwhile, when dynamic boundary conditions are regarded as a zero-thickness limit of certain transmission problems, third type conditions that characterize the permeability between the bulk domain and the thin region in transmission problems have also garnered significant interest.
\smallskip

In this paper, following the previous work in \cite{CFL19}, we consider a problem where the Cahn--Hilliard system is imposed in the interior domain and the Allen--Cahn type 
equation is prescribed on the boundary. 
In particular, we connect these equations by a third type transmission condition and perform an asymptotic analysis with respect to the parameter representing the permeability.
As is well known, the Cahn--Hilliard equation \cite{CH58} 
is a partial differential equation that describes phase separation, known as spinodal decomposition, and has traditionally  been applied to model various phenomena beyond phase separation in metal alloys. 
Similarly, the Allen--Cahn equation \cite{AC79} is a partial differential equation that describes solid-liquid phase transitions and has also contributed to modeling various phenomena. 
\smallskip

The Cahn--Hilliard system and the Allen--Cahn equation continue to attract broad research interest due to their versatility as mathematical models for a wide range of phenomena. 
Since dynamic boundary conditions can be interpreted as the zero-thickness limit of transmission problems coupling a bulk domain with a thin layer, it is therefore of considerable importance --- as carried out in this work --- to perform an asymptotic analysis of the associated permeability parameter.
\smallskip

We conclude the introduction with an outline of the paper. 
In Section~2, we first describe the initial-boundary value problem for the target partial differential equations, 
and then state the well-posedness of the intermediate problem, which serves as the starting point of the asymptotic analysis developed throughout this paper, 
in Theorem~\ref{lambda}. 
Although Theorem~\ref{lambda} is proved in Section~3 by applying the general theory of evolution equations governed by the subdifferential operators, 
we first obtain a characterization of the subdifferential. 
In Section~3, we derive uniform estimates, making explicit the dependence on $\lambda \in (0,1]$ 
and $K > 0$. 
In particular, since we perform two limiting procedures, $K \to +\infty$ and $K \to 0$, 
the reader should take care not to confuse the two. 
After proving Theorem~2.1, we consider in Section~4 the limiting procedures in the order $K \to +\infty$
and $K \to 0$. For the limit $K \to +\infty$, we substitute $\alpha := 1/K$ 
and study the limiting procedure as $\alpha \to 0$. 
Indeed, by rewriting the parameter in the third type transmission condition, 
the parameter $\alpha$ can be interpreted as a permeability parameter, 
and it becomes easier to visualize the situation in which the permeability tends to zero, 
corresponding to a state in which the bulk and boundary partial differential equations are completely decoupled. 
We note that the well-posedness of the completely split system is expected to admit a simpler proof than the one employed in this paper. 
In this sense, the significant aspect of Theorem~4.1 lies in the assertion concerning the rate of convergence as $\alpha \to 0$. The main theorem of this paper is Theorem~4.2, 
which shows that, through the asymptotic analysis as the permeability $\alpha \to +\infty$, 
equivalently $K \to 0$, the intermediate problem converges in a certain sense to the Cahn--Hilliard system equipped with Allen--Cahn type dynamic boundary conditions (cf.~\cite{CF15pier, CGS14}).

\section{Mathematical Formulation of the Problem}
\label{math}
\setcounter{equation}{0} 

\subsection{Cahn--Hilliard systems with Allen--Cahn type boundary equations}
Let $T>0$, $\Omega \subset \mathbb{R}^d$ ($d=2,3$) be the bounded domain with smooth boundary $\Gamma:=\partial \Omega$. 
For constant $K>0$, let us consider the following initial and boundary value problems: 
Find $u, \xi, \mu: (0,T) \times \Omega \to \mathbb{R}$ satisfying 
\begin{align}
	\partial _t u - \Delta \mu = 0 & \quad  {\rm in~} Q:=(0,T) \times \Omega, \label{CH1}\\
	-\Delta u + \xi + \pi (u) - f = \mu, \quad \xi \in \beta(u) & \quad {\rm in~} Q, \label{CH2}\\
	\partial_{\boldsymbol{ \nu }} \mu =0 & \quad {\rm on ~} \Sigma := (0,T) \times \Gamma, \label{NBC}\\
	\partial_{\boldsymbol{\nu}} u = \frac{1}{K} (\phi - u ) & \quad {\rm on~} \Sigma, \label{RBC}\\
	u(0) = u_0 & \quad {\rm in~}\Omega, \label{IC1}
\end{align}
along with $\phi, \psi : (0,T) \times \Gamma \to \mathbb{R}$ satisfying
\begin{align}
	\partial _t \phi - \Delta_\Gamma \phi + \psi + \pi_\Gamma (\phi) + \frac{1}{K}(\phi - u) =f_\Gamma,\quad \psi \in \beta_{\Gamma} (\phi) & \quad {\rm on~} \Sigma,\label{DBC}
	\\
	\phi(0) = \phi_0 & \quad {\rm on~}\Gamma, \label{IC2}
\end{align}
where, $\partial_t, \Delta, \partial_{\boldsymbol{\nu}}, \Delta_\Gamma$ stand for the time derivative, Laplace operator, normal derivative with respect to the normal vector $\boldsymbol{\nu}$, and Laplace--Beltrami operator, respectively; 
$f:Q \to \mathbb{R}$, 
$u_0:\Omega \to \mathbb{R}$, and 
$f_\Gamma: \Sigma \to \mathbb{R}$,  
$\phi_0: \Gamma \to \mathbb{R}$ are given functions; 
$\beta, \beta_\Gamma:\mathbb{R} \to 2^{\mathbb{R}}$ are 
possibly multivalued functions, and $\pi, \pi_\Gamma : \mathbb{R} \to \mathbb{R}$ are 
Lipschitz continuous functions, and together their sums construct the derivative of 
double well potentials in the bulk and boundary, respectively. 
The prototype setting is $\beta(r)+\pi(r)=\beta_\Gamma(r)+\pi_\Gamma(r):=r^3-r$ for $r \in \mathbb{R}$. 
Nevertheless, due to their multivalued nature, 
they can serve as nonlinear terms capable of describing various nonlinear phenomena. 
From \eqref{CH1}, \eqref{NBC}, and \eqref{IC1} we have the mean value conservation in the bulk 
\begin{equation}
	\frac{1}{|\Omega|} \int_\Omega u(t) \dx = \frac{1}{|\Omega|} \int_\Omega u_0 \dx =:m_0 
	\quad {\rm for~all~} t\in [0,T]. \label{const}
\end{equation} 
\smallskip

This problem is originally introduced in \cite[Subsection 2.4]{CFL19}. 
The crucial difference from other problems with dynamic boundary conditions lies in the terms appearing in equations \eqref{DBC} and \eqref{RBC}. 
Merging these terms leads to the well-known 
boundary equation of Allen--Cahn type
\begin{equation}
\label{pier1}
\partial _t \phi - \Delta_\Gamma \phi + \psi + \pi_\Gamma (\phi) + \partial_{\boldsymbol{\nu}} u  =f_\Gamma 
\quad \hbox{ with } \quad   \psi \in \beta_{\Gamma} (\phi) \quad \hbox{ on } \Sigma.
\end{equation}
On the other hand, while many existing studies of the dynamic boundary conditions impose the trace condition 
\begin{equation}
	u=\phi \quad {\rm on~} \Sigma \label{TC} 
\end{equation}
from the outset, 
we instead prescribe the third type transmission condition \eqref{RBC} with coefficient $\alpha:=1/K$, 
which allows a gap between the trace of the bulk order parameter $u$ and the boundary value $\phi$. In this formulation, 
the transmission condition and the trace condition are not assumed but rather recovered in the limit $K \to 0$. 
In other words, $\alpha=1/K$ can be regarded as a parameter that determines the permeability between the bulk and the surface (see, e.g., \cite{Att84, CP20, Fuk26}). 
\smallskip

As a remark, in the case of the dynamic boundary condition of Cahn--Hilliard type, this kind of intermediate situation is known as the KLLM model\cite{KLLM21} that serves to interpolate between the GMS 
model \cite{Gal06, GMS11, CF15b} and the LW model \cite{LW19, CFW20}. The GMS model corresponds to $K \to 0$ (permeability $\alpha \to +\infty$, namely it is fully permeable between the bulk and surface) 
and \eqref{TC} holds with \eqref{pier1}. On the other hand, the LW model corresponds to 
$K \to +\infty$ ($\alpha \to 0$ and the system is completely split). We refer to the review article \cite{Wu22} for an overview of mathematical results concerning Cahn--Hilliard systems with dynamic boundary conditions.

\subsection{Well-posedness for third type transmission problems}
Hereafter, we assume:
\begin{enumerate}
\item[(A1)] $\beta$, $\beta_\Gamma: \mathbb{R} \to 2^\mathbb{R}$ are maximal monotone graphs, 
that is, they are characterized by the subdifferentials $\beta = \partial \hat{\beta}$ and 
$\beta_\Gamma =\partial \hat{\beta}_\Gamma$, where 
$\hat{\beta}$, $\hat{\beta}_\Gamma : \mathbb{R} \to [0,+\infty]$ are 
lower semi-continuous and convex functions with $\hat{\beta}(0)=0$ and $\hat{\beta}_\Gamma(0)=0$; 
\item[(A2)] $\pi, \pi_\Gamma:\mathbb{R} \to \mathbb{R}$ are Lipschitz continuous functions; 
\item[(A3)] $f \in L^2(0,T;H^1(\Omega))$ and $f_\Gamma \in L^2(0,T;L^2(\Gamma))$; 
\item[(A4)] $u_0 \in H^1(\Omega)$ and $\phi_0 \in H^1(\Gamma)$; 
\item[(A5)] $D(\beta_\Gamma) \subset D(\beta)$ and there exists a constant $c_0>0$ such that 
\begin{equation}
	\hat{\beta}(r) \le c_0\bigl( 1+|r|^2+ \hat{\beta}_\Gamma(r) \bigr)
	\quad {\rm for~all~} r \in \mathbb{R};
	\label{gc}
\end{equation}
\item[(A6)] $m_0 \in {\rm int}D(\beta)$, ${\rm int}D(\beta_\Gamma) \ne \emptyset$, and the compatibility conditions 
$\hat{\beta}(u_0) \in L^1(\Omega)$, $\hat{\beta}_\Gamma (\phi_0) \in L^1(\Gamma)$ 
hold.
\end{enumerate}

\smallskip
The assumption (A1) implies that $0 \in \beta(0)$ and $0 \in \beta_\Gamma(0)$. 
Examples of $\beta,\,\pi$ and $\beta_\Gamma ,\, \pi_\Gamma$ satisfying these assumptions are given as follows, formulating them in terms of $\beta_\Gamma ,\, \pi_\Gamma \,$:
\smallskip
\begin{itemize}
\item $\beta_\Gamma(r):=r^3$, $\pi_\Gamma(r):=-r$ for $r \in \mathbb{R}$ (corresponding to the smooth double-well potential);\smallskip
\item $\beta_\Gamma(r):=\ln((1+r)/(1-r))$, $\pi_\Gamma(r):=-2cr$ for $r \in (-1,1)$ (derived from the singular potential of logarithmic type, where $c>0$ is a sufficiently large constant which breaks monotonicity);\smallskip
\item $\beta_\Gamma(r):=\partial I_{[-1,1]}(r)$, $\pi_\Gamma(r):=-r$ for $r \in [-1,1]$ (for the non-smooth potential, in this case $I_{[-1,1]}$ is the indicator function of the interval  $[-1,1] = D(\beta_\Gamma)$);\smallskip
\item $\beta_\Gamma(r):=0$, $\pi_\Gamma (r):=0$ for $r \in \mathbb{R}$ (reducing then to a linear equation). \smallskip
\end{itemize}
The possible choices of $\beta, \, \pi$ are similar to those of $\beta_\Gamma , \, \pi_\Gamma$ above, but the choices may be independent the ones from the others. Indeed, an important point is that it is possible to take different nonlinear monotone terms $\beta$ in the bulk 
and $\beta_\Gamma$ on the boundary, provided the related convex functions $\hat{\beta},\, \hat{\beta}_\Gamma$ fulfill the joint growth condition \eqref{gc}. 
For instance, if just $\hat{\beta}_\Gamma$ dominates $\hat{\beta}$ as growth, a possible 
combination is $\hat{\beta}(r):=r^4 / 4$ and $\hat{\beta}_\Gamma(r):= I_{[-1,1]}(r)$; 
conversely, if $\hat{\beta}$ has simply a quadratic growth, a choice like  $\hat{\beta}_\Gamma(r)=0$ is also admissible. The reader may also think of other possible cases. 
\smallskip

First, we establish the well-posedness of the intermediate problem under the above conditions as follows.

\begin{theorem}
\label{lambda} Assume {\rm (A1)}--{\rm (A6)}. Then, for each $K>0$ there exists a quintuplet
$(u,\xi, \mu,\phi, \psi)$ in these classes 
\begin{gather}
	u \in H^1 \bigl( 0,T;H^1(\Omega)' \bigr) \cap C\bigl( [0,T]; L^2(\Omega) \bigr) \cap L^\infty \bigl( 0,T ; H^1(\Omega) \bigr) \cap L^2 \bigl( 0,T ; H^2(\Omega) \bigr), \label{clu}\\
	\xi \in L^2 \bigl( 0,T ; L^2(\Omega) \bigr), \label{clx} \\ 
	\mu \in  L^2 \bigl( 0,T ; H^1(\Omega) \bigr), \label{clm}\\
	\phi \in H^1 \bigl( 0,T;L^2(\Gamma) \bigr) \cap C\bigl( [0,T]; H^1(\Gamma) \bigr) \cap L^2 \bigl( 0,T ; H^2(\Gamma) \bigr), \label{clphi}\\
	\psi \in L^2 \bigl( 0,T ; L^2(\Gamma) \bigr) \label{clpsi}
\end{gather}
such that 
\begin{align}
	\langle \partial _t u, y \rangle_{H^1(\Omega)', H^1(\Omega)} +\int_\Omega \nabla \mu \cdot \nabla y \dx = 0 \quad {\it for~all~} y \in H^1(\Omega), & \quad {\it a.e.\ on~}(0,T), \label{ch1}\\
	-\Delta u + \xi + \pi (u) - f  = \mu, \quad \xi \in \beta(u) & \quad {\it a.e.\ in~} Q, \label{ch2}\\
	\partial_{\boldsymbol{\nu}} u = \frac{1}{K} (\phi - u ) & \quad {\it a.e.\ on~} \Sigma, \label{rbc}\\
	u(0) = u_0 & \quad {\it a.e.\ in~}\Omega, \label{ic1}\\
	\partial _t \phi - \Delta_\Gamma \phi + \psi + \pi_\Gamma (\phi) + \frac{1}{K}(\phi - u) =f_\Gamma,\quad \psi \in \beta_{\Gamma} (\phi) & \quad {\it a.e.\ on~} \Sigma,\label{dbc}
	\\
	\phi(0) = \phi_0 & \quad {\it a.e.\ on~}\Gamma. \label{ic2}
\end{align}
Moreover, $u$, $\phi$, and $\psi$ are uniquely determined. 
If $\beta$ is single-valued, then $\mu$ and $\xi$ are also uniquely determined. 
\end{theorem}

In this theorem, the equation \eqref{CH1} with the homogeneous Neumann boundary condition \eqref{NBC} is 
interpreted by \eqref{ch1}, in other words, the boundary condition for $\mu$ is hidden in the weak formulation.

\subsection{Application of abstract theory to approximate problem}

To discuss the existence and uniqueness, 
we employ the abstract theory of evolution equations 
governed by subdifferential operators. 
Firstly, we recall the concept of 
Yosida approximation to $\beta$ and $\beta_\Gamma$. 
Let $\lambda \in (0,1]$ be the approximation parameter. 
The Yosida approximation of $\beta$ is defined by 
\begin{equation*}
	\beta_\lambda (r):=\frac{1}{\lambda} \bigl( r-J_\lambda (r) \bigr) :=\frac{1}{\lambda} 
	\bigl( r-(I+\lambda \beta) ^{-1}(r) \bigr) \quad {\rm for~} r \in \mathbb{R},
\end{equation*}
where $J_\lambda:=(I+\lambda \beta) ^{-1}: \mathbb{R} \to \mathbb{R}$ is called the resolvent of $\beta$ (see, e.g., \cite{Bar10, Bre73}). 
From the theory of maximal monotone operators, we also know that $\beta_\lambda$ is characterized by 
the Moreau--Yosida regularization $\hat{\beta}_\lambda :\mathbb{R} \to \mathbb{R}$ defined by 
\begin{equation*}
	\hat{\beta}_\lambda(r) := \inf_{s \in \mathbb{R}} \left\{ \frac{1}{2\lambda} |r -s |^2 + \hat{\beta} (s) \right\}
	= \frac{1}{2\lambda} \bigl|r - J_\lambda (r) \bigr|^2 + \hat{\beta} \bigl( J_\lambda (r) \bigr) \quad 
	{\rm for~} r \in \mathbb{R},
\end{equation*} 
and such that $\beta_\lambda = \hat{\beta}_\lambda^{\, \prime}$.
Then, we see that
\begin{equation}
\label{pier11}
0 \le \hat{\beta}_\lambda(r) \le \hat {\beta}(r) \quad \hbox{for all $r \in \mathbb{R}$}. 
\end{equation}
We define the approximations $\beta_{\Gamma,\lambda}$ and $\hat{\beta}_{\Gamma,\lambda}$ to $\beta_\Gamma$ and $\hat{\beta}_\Gamma$, respectively, in a similar fashion. 
\smallskip 

Next, we introduce another approximation to the Cahn--Hilliard system. 
Adding $\lambda \mu$ to \eqref{ch1}, and $\lambda \partial _t u$ to \eqref{ch2}
we temporarily replace the target equations \eqref{ch1}--\eqref{ch2} 
with the following equations.
\begin{align*}
	\partial _t u  + \lambda \mu -\Delta \mu =0, & \\ 
	\lambda \partial _t u - \mu -\Delta u =f - \beta_\lambda(u) - \pi (u)& 
	\quad {\rm a.e.\ in~} Q. 
\end{align*}
Here, using the homogeneous Neumann Laplacian $\Delta _{\rm N}$ we rewrite the first equation with \eqref{NBC} to 
$-\mu = (\lambda I-\Delta _{\rm N})^{-1} \partial _t u$. 
We define a convex functional $\Phi: L^2(\Omega) \times L^2(\Gamma) \to [0,+\infty]$
as follows
\begin{equation*}
	\Phi (u, \phi)
	:= 
	\begin{cases}
	\displaystyle 
	\frac{1}{2} \int_{\Omega}|\nabla u|^2 \dx 
	+ 
	\frac{1}{2} \int_\Gamma |\nabla_\Gamma \phi|^2 \dg 
	+ \frac{1}{2K} \int_\Gamma |u-\phi|^2 \dg 
	& {\rm if~} (u,\phi) \in D(\Phi), \\
	+\infty & {\rm otherwise},
	\end{cases}
\end{equation*}
where $D(\Phi):=H^1(\Omega) \times H^1(\Gamma)$ is independent of $K>0$. Therefore,  
our first target problem is the Cauchy problem of the following evolution equation:
\begin{equation}
	\begin{pmatrix}
	[\lambda I + (\lambda I - \Delta _{\rm N})^{-1}] u'(t) \\
	\phi'(t)
	\end{pmatrix}
	+ \partial \Phi 
	\begin{pmatrix}
	u(t) \\
	\phi(t)
	\end{pmatrix}
	= \begin{pmatrix}
	-\beta_\lambda \bigl( u(t) \bigr) - \pi \bigl( u(t) \bigr) +f(t) \\
	-\beta_{\Gamma,\lambda} \bigl( \phi(t) \bigr) - \pi_\Gamma \bigl( \phi(t) \bigr) +f_\Gamma(t) 
	\end{pmatrix}\label{system}
\end{equation} 
for a.a.\ $t \in (0,T)$ with 
the initial condition 
$(u(0),\phi(0))^{\sf T}=(u_0,\phi_0)^{\sf T}$ 
in $L^2(\Omega) \times L^2(\Gamma)$. 
Hereafter, we do not care about the row or the column so we omit the symbol ${\sf T}$. 
\smallskip
 
From the above, it can be seen that solving \eqref{system} amounts to solving an approximate problem for our problem \eqref{ch1}--\eqref{ic2}; this will become clear by the following characterization of the subdifferential.

\begin{lemma}\label{Lemmacha}
Let $(y,\upsilon) \in D(\Phi)$. The pair $(y^*,\upsilon^*)$ belongs to the subdifferential $\partial \Phi (y,\upsilon)$ in $L^2(\Omega) \times L^2(\Gamma)$ if and only if 
\begin{align}
	y^* = -\Delta y & \quad {\it a.e.~in~} \Omega, \label{cha1}\\
	\partial_{\boldsymbol{\nu}} y = \frac{1}{K} (\upsilon - y) & \quad {\it a.e.~on~} \Gamma, \label{cha2}\\
	\upsilon^* = -\Delta_\Gamma \upsilon + \frac{1}{K}( \upsilon - y) & \quad {\it a.e.~on~} \Gamma.
	\label{cha3}
\end{align}
This implies that $\partial \Phi (y,\upsilon)$ is a singleton and 
\begin{equation}
	D(\partial \Phi) = \left\{ (z,\zeta ) \in H^2(\Omega) \times H^2(\Gamma) : 
	\partial_{\boldsymbol{\nu}} z = \frac{1}{K} (\zeta - z) 
	\ {\it a.e.\ on~} \Gamma
  \right\}. \label{domain}
\end{equation}
\end{lemma}

\begin{proof} Let $(y,\upsilon) \in D(\Phi)$ and $(y^*,\upsilon^*) \in \partial \Phi (y,\upsilon)$ in $L^2(\Omega) \times L^2(\Gamma)$. From the definition of the subdifferential we have 
\begin{align*}
	& (y^*,z-y)_{L^2(\Omega)} + (\upsilon^*,\zeta-\upsilon)_{L^2(\Gamma)} 
	\\
	& \le \frac{1}{2} \int_{\Omega}|\nabla z|^2 \dx 
	+ 
	\frac{1}{2} \int_\Gamma |\nabla_\Gamma \zeta|^2 \dg 
	+ \frac{1}{2K} \int_\Gamma |z-\zeta|^2 \dg \\
	& \quad {}
	-\frac{1}{2} \int_{\Omega}|\nabla y|^2 \dx 
	-
	\frac{1}{2} \int_\Gamma |\nabla_\Gamma \upsilon|^2 \dg 
	- \frac{1}{2K} \int_\Gamma |y-\upsilon|^2 \dg 
	\quad {\rm for~all~} (z,\zeta) \in L^2(\Omega) \times L^2(\Gamma).  
\end{align*}
For each $\varepsilon>0$ and $(z,\zeta) \in D(\Phi)$, taking 
$y+\varepsilon z$ as $z$, and $\upsilon+\varepsilon \zeta$ as 
$\zeta$ in the above, then 
\begin{align*}
	& (y^*,\varepsilon z)_{L^2(\Omega)} + (\upsilon^*,\varepsilon \zeta)_{L^2(\Gamma)} 
	= \bigl( y^*,(y+\varepsilon z) -y \bigr)_{L^2(\Omega)} + \bigl( 
	\upsilon^*,(\upsilon+\varepsilon \zeta) - \upsilon \bigr)_{L^2(\Gamma)} 
	\notag \\
	& \le \Phi \bigl( (y+\varepsilon z), (\upsilon+\varepsilon \zeta)\bigr) - \Phi(y, \upsilon)
	\notag 
	\\
	& \le 
\varepsilon \int_\Omega \nabla y \cdot \nabla z \dx + \frac{\varepsilon^2}{2} \int_\Omega |\nabla z|^2 \dx 
	+\varepsilon \int_\Gamma \nabla_\Gamma \upsilon \cdot \nabla_\Gamma \zeta \dg + \frac{\varepsilon^2}{2} \int_\Gamma |\nabla_\Gamma \zeta |^2 \dg \\
	& \quad {} + \frac{\varepsilon}{K} \int_\Gamma (y-\upsilon)(z-\zeta) \dg 
	+ \frac{\varepsilon^2}{2K} \int_\Gamma |z-\zeta|^2 \dg. 
\end{align*}
Dividing by $\varepsilon>0$ and letting $\varepsilon \to 0$ we obtain 
\begin{align*}
	& 
	(y^*,z)_{L^2(\Omega)} + (\upsilon^*,\zeta)_{L^2(\Gamma)} 
	\\
	& \le \int_\Omega \nabla y \cdot \nabla z \dx 
	+
	\int_\Gamma \nabla_\Gamma \upsilon \cdot \nabla_\Gamma \zeta \dg
	+ \frac{1}{K} \int_\Gamma (y-\upsilon)(z-\zeta) \dg
	\quad {\rm for~all~} (z,\zeta) \in D(\Phi).
\end{align*}
The opposite inequality can be also obtained by taking $y-\varepsilon z$ as $z$, 
and $\upsilon-\varepsilon \zeta$ as 
$\zeta$. Therefore, we conclude that $(y^*,\upsilon^*)$ satisfies
\begin{align}
	& 
	(y^*,z)_{L^2(\Omega)} + (\upsilon^*,\zeta)_{L^2(\Gamma)} 
	\notag
	\\
	& = \int_\Omega \nabla y \cdot \nabla z \dx 
	+
	\int_\Gamma \nabla_\Gamma \upsilon \cdot \nabla_\Gamma \zeta \dg
	+ \frac{1}{K} \int_\Gamma (y-\upsilon)(z-\zeta) \dg
	\label{weak}
\end{align}
for all $(z,\zeta) \in D(\Phi)$. 
\smallskip

Next, we provide the detail of characterizations \eqref{cha1}--\eqref{cha3}, step by step. 
Firstly, take 
$z \in {\mathcal D}(\Omega)$ and $\zeta \equiv 0$ in \eqref{weak}. 
Then, using the condition $z-\zeta =0$ a.e.\ on $\Gamma$ we find that 
\begin{equation*}
	y ^* = -\Delta y \quad {\rm in~} {\mathcal D}'(\Omega);
\end{equation*}
additionally, we have already known that $y^* \in L^2(\Omega)$, 
therefore we obtain \eqref{cha1} by comparison in the equation. 
Secondly, let $w \in H^{1/2}(\Gamma)$ and 
fix a linear recovery operator ${\mathcal R}:H^{1/2}(\Gamma) \to H^1(\Omega)$ such that 
$\gamma ({\mathcal R}w)= w$, where $\gamma:H^1(\Omega) \to H^{1/2}(\Gamma)$  denotes
the trace operator.
Take $z={\mathcal R}w \in H^1(\Omega)$ and $\zeta \equiv 0$ in \eqref{weak}. Then, 
using the generalized Green formula \cite[Corollary 2.6]{GR86}
and \eqref{weak} we have 
\begin{align}
	&\langle \partial _{\boldsymbol{\nu}} y, w \rangle_{H^{-1/2}(\Gamma),H^{1/2}(\Gamma)}
	 = \int_\Omega \Delta y {\mathcal R} w \dx 
	+ \int_\Omega \nabla y \cdot \nabla {\mathcal R} w \dx \notag \\
	& = -(y^*,{\mathcal R} w)_{L^2(\Omega)} + 
	\int_\Omega \nabla y \cdot \nabla {\mathcal R} w \dx 
	= \frac{1}{K} \int_\Gamma (\upsilon-y )w  \dg, \label{Green}
\end{align}
that is, \eqref{cha2} holds. 
Of course, 
that reasoning uses the fact that the equality \eqref{cha2} first holds in the $H^{-1/2}(\Gamma)$ sense, 
and subsequently, from comparison of terms and elliptic regularity, the equality holds in the $L^2(\Gamma)$ sense. Finally, take 
$z \in H^1(\Omega)$ and $\zeta \in H^1(\Gamma)$. Then, using \eqref{cha1}, \eqref{cha2}, and 
\eqref{weak} we deduce 
\begin{align*}
	& 
	(\upsilon^*,\zeta)_{L^2(\Gamma)} 
	\notag
	\\
	& = -(y^*,z)_{L^2(\Omega)} + \int_\Omega \nabla y \cdot \nabla z \dx 
	+
	\int_\Gamma \nabla_\Gamma \upsilon \cdot \nabla_\Gamma \zeta \dg
	+ \frac{1}{K} \int_\Gamma (y-\upsilon)(z-\zeta) \dg \\
	& = \int_\Gamma \partial _{\boldsymbol{\nu}} y z \dg 
	+ \int_\Gamma \nabla_\Gamma \upsilon \cdot \nabla_\Gamma \zeta \dg
	+ \frac{1}{K} \int_\Gamma (y-\upsilon)z \dg 
	- \frac{1}{K} \int_\Gamma (y-\upsilon)\zeta \dg \\
	& = \int_\Gamma \nabla_\Gamma \upsilon \cdot \nabla_\Gamma \zeta \dg
	+ \frac{1}{K} \int_\Gamma (\upsilon-y)\zeta \dg. 
\end{align*}
Thus, \eqref{cha3} can be obtained from the comparison in the equation, again. 
For $(y,\upsilon) \in H^1(\Omega) \times H^1(\Gamma)$, 
the regularities of $\Delta y \in L^2(\Omega)$ and $\partial_{\boldsymbol{\nu}} y =(1/K)(\upsilon-y)\in H^{1/2}(\Gamma)$ guarantees that $y \in H^2(\Omega)$ 
via the standard elliptic estimate (e.g., \cite[Theorem~3.2, p.~1.79]{BG87})
\begin{equation}
	\| y \|_{H^2(\Omega)}
	\le C_{\rm E}\bigl( \| \Delta  y \|_{L^2(\Omega)} + \| \partial_{\boldsymbol{\nu}} y \|_{H^{1/2}(\Gamma)}
	 + \| y \|_{L^2(\Omega)}  \bigr), 
	\label{eli}
\end{equation}
where $C_{\rm E}$ is a positive constant. 
Moreover, the property
$\Delta_\Gamma \upsilon \in L^2(\Gamma)$ ensures that $\upsilon \in H^2(\Gamma)$. 
Therefore, we conclude \eqref{domain}. 
\smallskip

To prove the inverse, that is, to prove $(y^*,\upsilon^*) \in \partial \Phi(y,\upsilon)$ from  \eqref{cha1}--\eqref{domain}, 
we use the definition of the subdifferential and the integration by part formula. 
The proof is very simple and thus omitted. 
\end{proof}
\smallskip

\begin{remark} In {\rm Lemma \ref{Lemmacha}}, we can find the following property 
\begin{equation}
	\int_\Omega y^* \dx + \int_\Gamma \upsilon^* \dg =0
	\label{remark}
\end{equation}
for the element $(y^*, \upsilon^*)$ of the subdifferential $\partial \Phi(y,\upsilon)$.  
Indeed, integrating \eqref{cha3} over $\Gamma$ and using \eqref{cha2} and \eqref{cha1}
lead to \eqref{remark}.
It can be also obtained in another way. More precisely, taking $z\equiv 1$ and $\zeta \equiv 1$ in 
\eqref{weak}, we get \eqref{remark}. 
The property \eqref{remark} is hidden in \eqref{cha1}--\eqref{cha3}. 
\end{remark}

Based on the abstract theory of doubly nonlinear evolution equation \cite{CV90} and Lipschitz perturbation treatment (see, e.g., \cite{CF15a, CF15b}), we can prove the existence of
\begin{equation*}
	(u_\lambda,\phi_\lambda) \in H^1\bigl( 0,T ; L^2(\Omega)\times L^2(\Gamma) \bigr) \cap C\bigl( [0,T]; D(\Phi) \bigr) 
	\cap L^2(0,T;D(\partial \Phi)\bigr),
\end{equation*}
which is the unique strong solution to the system \eqref{system}, where 
we needed the given data: 
$f \in L^2(0,T;L^2(\Omega))$, $f_\Gamma \in L^2(0,T;L^2(\Gamma))$, 
$u_0 \in H^1(\Omega)$, and $\phi_0 \in H^1(\Gamma)$. These follow from assumptions {\rm (A3)} and {\rm (A4)}. As a remark, for this result we do not need any compatibility condition between $u_0$ and $\phi_0$, they are completely independent.  
\smallskip

To summarize the discussion presented above, we now rewrite the system in terms of the unique triplet $(u_\lambda, \mu_\lambda, \phi_\lambda)$ in
\begin{gather*}
	u_\lambda \in H^1\bigl( 0,T ; L^2(\Omega) \bigr) \cap C\bigl( [0,T]; H^1(\Omega) \bigr) 
	\cap L^2(0,T;H^2(\Omega) \bigr), \\
	\mu_\lambda \in L^2 \bigl(0,T;H^2(\Omega) \bigr), \\
	\phi_\lambda \in H^1\bigl( 0,T ; L^2(\Gamma) \bigr) \cap C\bigl( [0,T]; H^1(\Gamma) \bigr) 
	\cap L^2(0,T;H^2(\Gamma) \bigr)
\end{gather*}
satisfying 
\begin{align}
	\partial _t u_\lambda + \lambda \mu_\lambda - \Delta \mu_\lambda = 0 & \quad  {\rm a.e.\ in~} Q, \label{ch1lam}\\
	\lambda \partial _t u_\lambda -\Delta u_\lambda + \beta_\lambda(u_\lambda) + \pi (u_\lambda) - f = \mu_\lambda
	& \quad {\rm a.e.\ in~} Q, \label{ch2lam}\\
	\partial_{\boldsymbol{ \nu }} \mu_\lambda =0 & \quad {\rm a.e.\ on ~} \Sigma, \label{nbclam}\\
	\partial_{\boldsymbol{\nu}} u_\lambda = \frac{1}{K} (\phi_\lambda - u_\lambda ) & \quad {\rm a.e.\ on~} \Sigma, \label{rbclam}\\
		u_\lambda(0) = u_0 & \quad {\rm a.e.\ in~}\Omega, \label{ic1lam}\\
	\partial _t \phi_\lambda - \Delta_\Gamma \phi_\lambda + \beta_{\Gamma,\lambda}(\phi_\lambda) + \pi_\Gamma (\phi_\lambda) + \frac{1}{K}(\phi_\lambda - u_\lambda) =f_\Gamma & \quad {\rm a.e.\ on~} \Sigma,\label{dbclam}
	\\
	\phi_\lambda(0) = \phi_0 & \quad {\rm a.e.\ on~}\Gamma. \label{ic2lam}
\end{align}
In the next section, we are going to derive uniform estimates, i.e., independent of $\lambda \in (0,1]$ and~$K>0$. 

\section{Uniform estimates and limiting procedures}
\label{UandL}
\setcounter{equation}{0} 

In this section, we obtain uniform estimates for the approximate solution which is obtained in the previous section. To deal with the limiting procedure with respect to 
$\lambda \to 0$ first, $K \to +\infty$ and $K \to 0$ second, we take care of the dependence of their parameters. 
\smallskip

At first, we define $m : H^1(\Omega)' \to \mathbb{R}$ by
\begin{equation}
\label{pier2}
m(z):=  \frac 1{|\Omega|} \langle z, 1 \rangle_{H^1(\Omega)', H^1(\Omega)} \quad \hbox{for } z \in H^1(\Omega)'. 
\end{equation}
If $z \in L^1(\Omega)$, then $m(z)$ may be expressed as $ (1/|\Omega|)\int_\Omega z \dx$. 
From the very beginning, we have known that
\begin{equation}
	m\bigl( u_\lambda (t) \bigr) 
	+ \lambda \int_0^t m \bigl( \mu_\lambda( \tau ) \bigr) \d \tau = m_0 
	\quad {\rm for~all~} t \in [0,T]. \label{volume}
\end{equation}
Indeed, integrate \eqref{ch1lam} over $(0,t) \times \Omega$ with respect to the time and space variables, and use \eqref{nbclam} along with \eqref{const}. 

\begin{lemma}\label{L1} There exist positive constants $M_1$ and $M_2$, independent of $\lambda \in (0,1]$ and 
$K>0$, such that
\begin{align}
	&\sqrt{\lambda} \| \mu_\lambda\|_{L^2(0,T;L^2(\Omega))}
	+ \| \nabla \mu_\lambda\|_{L^2(0,T;L^2(\Omega))}
	+\sqrt{\lambda}  \| \partial_t u_\lambda\|_{L^2(0,T;L^2(\Omega))} 
	+ \| u_\lambda\|_{L^\infty(0,T;H^1(\Omega))} 
	\notag \\
	& {}
	+ \bigl\| \hat{\beta}_\lambda(u_\lambda) \bigr\|_{L^\infty(0,T;L^1(\Omega))}^{1/2}
	+ \frac{1}{\sqrt{K}} \| u_\lambda-\phi_\lambda \|_{L^\infty(0,T;L^2(\Gamma))} 
	+ \| \phi_\lambda \|_{H^1(0,T;L^2(\Gamma))}
	\notag \\
	& {}
	+ \| \phi_\lambda \|_{L^\infty(0,T;H^1(\Gamma))}
	+ \bigl\| \hat{\beta}_{\Gamma,\lambda} (\phi_\lambda) \bigr\|_{L^\infty(0,T;L^1(\Gamma))}^{1/2}
	\le M_1
	\left(1
	+ \frac{1}{\sqrt{K}} \| u_0-\phi_0 \|_{L^2(\Gamma)} \right),
	\label{est1}
	\\
	&\| \partial _t u_\lambda \|_{L^2(0,T;H^1(\Omega)')} \le M_2
	\left(1
	+ \frac{1}{\sqrt{K}} \| u_0-\phi_0 \|_{L^2(\Gamma)} \right).
	\label{est2}
\end{align}
\end{lemma}

\begin{proof} Multiplying \eqref{ch1lam} by $\mu_\lambda$ and \eqref{ch2lam} by $\partial_t u_\lambda$, 
integrating them over $\Omega$ with respect to the space variable, and adding them we notice a cancellation of two terms and infer that
\begin{align}
	& \lambda \| \mu_\lambda \|_{L^2(\Omega)}^2 + \int_\Omega |\nabla \mu_\lambda|^2 \dx 
	+ \lambda \| \partial_t u_\lambda \|_{L^2(\Omega)}^2 
	+ \frac{1}{2} \frac{\d}{\dt} \| u_\lambda\|_{H^1(\Omega)}^2 
	\notag \\
	& \quad {}
	+ \frac{1}{K} \int_\Gamma (u_\lambda-\phi_\lambda) \partial_t u_\lambda \dg 
	+ \frac{\d}{\dt} \int_\Omega \hat{\beta}_\lambda(u_\lambda) \dx
	\notag \\
	& = \bigl( u_\lambda - \pi(u_\lambda) + f, \partial _t u_\lambda \bigr)_{L^2(\Omega)} \notag \\
	& = \bigl( u_\lambda - \pi(u_\lambda) + f, -\lambda \mu_\lambda \bigr)_{L^2(\Omega)} 
	- \int_\Omega \nabla \bigl( u_\lambda - \pi(u_\lambda) + f \bigr)\cdot \nabla \mu_\lambda \dx \notag \\
	& \le 
	\frac{1}{2}\bigl\| u_\lambda - \pi(u_\lambda) + f \bigr\|_{H^1(\Omega)}^2
	+ 
	\frac{\lambda^2}{2} \| \mu_\lambda \|_{L^2(\Omega)}^2 
	+ \frac{1}{2} \int _\Omega |\nabla \mu_\lambda|^2 \dx\label{case1}
\end{align}
a.e.\ in $(0,T)$, where we added the term $(u_\lambda, \partial _t u_\lambda)_{L^2(\Omega)}$ to both sides and 
used the equation \eqref{ch1lam} and conditions \eqref{nbclam}, \eqref{rbclam}. 
At the same time we test \eqref{dbclam} by $\partial_t \phi_\lambda$, 
integrate over $\Gamma$, and add the term $(\phi_\lambda, \partial _t \phi_\lambda)_{L^2(\Gamma)}$ to both sides, obtaining 
\begin{align}
	& \| \partial_t \phi_\lambda\|_{L^2(\Gamma)}^2 + \frac{1}{2} \frac{\d}{\dt} \| \phi_\lambda \|_{H^1(\Gamma)}^2  
	+ \frac{\d}{\dt} \int_\Gamma \hat{\beta}_{\Gamma,\lambda} (\phi_\lambda) \dg 
	+ \frac{1}{K} \int_\Gamma (\phi_\lambda-u_\lambda) \partial_t \phi_\lambda \dg 
	\notag\\ 
	& \le \frac{1}{2} \bigl\| \phi_\lambda -\pi_\Gamma (\phi_\lambda) + f_\Gamma \bigr\|_{L^2(\Gamma)}^2
	+ \frac{1}{2} \| \partial_t \phi_\lambda \|_{L^2(\Gamma)}^2.
	\label{add}
\end{align}
By adding \eqref{case1} and \eqref{add}, we proceed and note that, as $0<\lambda\leq 1$, there exists a positive constant $C_{\rm L}$ which depends on the Lipschitz constants of $\pi$ and $\pi_\Gamma$ such that 
\begin{align}
	& \frac{\lambda}{2} \| \mu_\lambda \|_{L^2(\Omega)}^2 
	+ \frac{1}{2}\int_\Omega |\nabla \mu_\lambda|^2 \dx 
	+ \lambda \| \partial_t u_\lambda \|_{L^2(\Omega)}^2 
	+ \frac{1}{2} \frac{\d}{\dt} \| u_\lambda\|_{H^1(\Omega)}^2 
	+ \frac{\d}{\dt} \int_\Omega \hat{\beta}_\lambda(u_\lambda) \dx
		\notag \\
	& \quad {}
	+ \frac{1}{2K} \frac{\d}{\dt}\| u_\lambda-\phi_\lambda \|_{L^2(\Gamma)}^2 
	+ \frac{1}{2}\| \partial_t \phi_\lambda\|_{L^2(\Gamma)}^2 
	+ \frac{1}{2} \frac{\d}{\dt}  \| \phi_\lambda \|_{H^1(\Gamma)}^2  + \frac{\d}{\dt} \int_\Gamma \hat{\beta}_{\Gamma,\lambda} (\phi_\lambda) \dg 
	\notag \\
	& \le 
	\frac{1}{2}\bigl\| u_\lambda - \pi(u_\lambda) + f \bigr\|_{H^1(\Omega)}^2
	+\frac{1}{2} \bigl\| \phi_\lambda -\pi_\Gamma (\phi_\lambda) + f_\Gamma \bigr\|_{L^2(\Gamma)}^2 
	\notag \\
	& \le C_{\rm L} \Bigl( 1 + \| u_\lambda \|_{H^1(\Omega)}^2 + \|f\|_{H^1(\Omega)}^2 \Bigr)
	+C_{\rm L} \Bigl( 1 + \| \phi_\lambda \|_{L^2(\Gamma)}^2 + \|f_\Gamma \|_{L^2(\Gamma)}^2 \Bigr).
	\label{back}
\end{align}
Therefore, the Gronwall inequality implies that
\begin{align*}
	&
	\frac{1}{2} \| u_\lambda\|_{H^1(\Omega)}^2 
	+ \int_\Omega \hat{\beta}_\lambda(u_\lambda) \dx
	+ \frac{1}{2K} \| u_\lambda-\phi_\lambda \|_{L^2(\Gamma)}^2 
	+ \frac{1}{2}  \| \phi_\lambda \|_{H^1(\Gamma)}^2 
	+ \int_\Gamma \hat{\beta}_{\Gamma,\lambda} (\phi_\lambda) \dg 
	\notag \\
	& \le M_1'
	\left(1
	+ \frac{1}{K} \| u_0-\phi_0 \|_{L^2(\Gamma)}^2  \right)
\end{align*}
on $[0,T]$, where the constant $M_1'>0$ depends on $\| u_0 \|_{H^1(\Omega)}$,  
$\| \hat{\beta}(u_0) \|_{L^1(\Omega)}$, 
$\| \phi_0 \|_{H^1(\Gamma)}$,  
$\| \hat{\beta}_\Gamma (\phi_0) \|_{L^1(\Gamma)}$, 
$\|f\|_{L^2(0,T;H^1(\Omega))}$, and 
$\|f_\Gamma \|_{L^2(0,T;L^2(\Gamma))}$. These quantities are finite by the assumptions {\rm (A3)}, 
{\rm (A4)}, and {\rm (A6)}. 
Going back to \eqref{back} we deduce \eqref{est1}. 
\smallskip

To obtain \eqref{est2}, we recall \eqref{ch1lam} and \eqref{nbclam} to get 
\begin{align*}
	\int_0^T \bigl| \langle \partial_t u_\lambda, \eta \rangle_{H^1(\Omega)'H^1(\Omega)} \bigr| \dt
	\le \lambda \int_0^T \| \mu_\lambda \|_{L^2(\Omega)} \|\eta \|_{L^2(\Omega)} \dt
	 + \int_0^T \| \nabla \mu_\lambda \|_{L^2(\Omega)} \| \eta \|_{H^1(\Omega)} \dt 
\end{align*}
for all $\eta \in L^2(0,T;H^1(\Omega))$. Thus, \eqref{est1} entails \eqref{est2} because $\sqrt{\lambda} \le 1$. 
\end{proof}

\begin{remark}
We can argue differently in the estimate \eqref{case1} for the term involving $f$. In fact, doing an integration by parts in time we have that
\begin{align*}
	&\int_0^t (f, \partial _t u_\lambda)_{L^2(\Omega)} \d\tau \\
	&= - \int_0^t \langle \partial _t f , u_\lambda \rangle_{H^1(\Omega)',H^1(\Omega)} \dtau 
	+ \bigl \langle f(t), u_\lambda(t)  \bigr\rangle_{H^1(\Omega)',H^1(\Omega)}  
	- \bigl\langle f(0), u_0 \bigr\rangle_{H^1(\Omega)',H^1(\Omega)}.
\end{align*}
This suggests that we may be able to replace the assumption
$ f \in L^2(0,T;H^1(\Omega))$ in {\rm (A3)} with $f \in W^{1,1}(0,T;H^1(\Omega)')$
and still obtain the uniform estimate~\eqref{back}. 
However, let us underline that the regularity $f \in L^2(0,T;L^2(\Omega))$ is required in the subsequent analysis.
\end{remark}

\begin{lemma}\label{L2} There exists a positive constant $M_3$, independent of $\lambda \in (0,1]$ and 
$K>0$, such that
\begin{align}
	& 
	\bigl\| \beta _\lambda (u_\lambda ) \bigr\|_{L^2(0,T;L^1(\Omega))} 
	+ \bigl\| \beta _{\Gamma,\lambda} (\phi_\lambda ) \bigr\|_{L^2(0,T;L^1(\Gamma))} 
	\notag \\
	& \le M_3 \left(  1 + \frac{1}{K} |m_0-m_0'|^2 + \frac{1}{K} \|u_0 - \phi_0 \|_{L^2(\Gamma)}^2 \right),
	\label{est3}
\end{align}
where $m_0' \in \mathbb {R}$ is an arbitrary value lying in the interior of $D(\beta_\Gamma)$.
\end{lemma}

\begin{proof} Multiplying \eqref{ch2lam} by $u_\lambda-m_0$, integrating it over $\Omega$, 
and using \eqref{rbclam}
we obtain
\begin{align}
	& \int_\Omega \beta_\lambda(u_\lambda) (u_\lambda-m_0) \dx 
	+ \int_\Omega |\nabla u_\lambda|^2 \dx
	\notag \\ 
	& = \frac{1}{K} \int_\Gamma (\phi_\lambda- u_\lambda)(u_\lambda-m_0) \dg 
	+ \int_\Omega \bigl( f- \pi(u_\lambda)-\lambda \partial _t u_\lambda \bigr)(u_\lambda-m_0) \dx 
	\notag \\
	& \quad {}+ \int_\Omega \mu_\lambda \bigl(
	u_\lambda-m_0 + \lambda (1*\mu_\lambda ) \bigr) \dx
	- \lambda \int_\Omega \mu_\lambda  (1*\mu_\lambda )  \dx,
	\label{add1}
\end{align}
where we employed the symbol 
\begin{equation*}
	(1 * \mu_\lambda) (t):= \int_0^t \mu_\lambda (\tau) \d \tau
	\quad {\rm for~} t \in [0,T]. 
\end{equation*}
Now, in view of the conservation property \eqref{volume} it holds that
\begin{align}
	& \int_\Omega \mu_\lambda \bigl(
	u_\lambda-m_0 + \lambda (1*\mu_\lambda ) \bigr) \dx
	\notag \\
	& = \int_\Omega\bigl( \mu_\lambda -m(\mu_\lambda) \bigr) 
	\bigl(
	u_\lambda-m_0 + \lambda (1*\mu_\lambda ) \bigr) \dx
	+ \int_\Omega m(\mu_\lambda) \bigl(
	u_\lambda-m_0 + \lambda (1*\mu_\lambda ) \bigr) \dx 
	\notag \\
	& = \int_\Omega \bigl( \mu_\lambda -m (\mu_\lambda) \bigr) (
	u_\lambda-m_0) \dx
	+ \lambda
	\int_\Omega \bigl( \mu_\lambda -m(\mu_\lambda) \bigr) \int_0^{(\cdot)}
	 \mu_\lambda (\tau) \d \tau  \dx
	\notag \\
	& \le  \bigl\| \mu_\lambda -m(\mu_\lambda) \bigr\|_{L^2(\Omega)} \| u_\lambda -m_0 \|_{L^2(\Omega)}  
	+ \lambda
	 \bigl\| \mu_\lambda -m(\mu_\lambda) \bigr\|_{L^2(\Omega)} \int_0^{(\cdot)}
	\bigl\| \mu_\lambda (\tau) \bigr\|_{L^2(\Omega)} \d \tau \notag \\
	& 
	\le C_{\rm P} \| \nabla \mu_\lambda \|_{L^2(\Omega)} \| u_\lambda -m_0 \|_{L^2(\Omega)}  
		+ \lambda C_{\rm P} \| \nabla \mu_\lambda \|_{L^2(\Omega)} 
	\int_0^{(\cdot)}
	\bigl\| \mu_\lambda (\tau) \bigr\|_{L^2(\Omega)} \d \tau,
	\label{add2}
\end{align}
a.e.\ in $(0,T)$, where we 
applied the Poincar\'e inequality in terms of the positive constant $C_{\rm P}$.
Similarly, a.e.\ in $(0,T)$ we have that
\begin{align}
	&-\lambda \int_\Omega \mu_\lambda  (1*\mu_\lambda )  \dx 
	 = -\int_0^{(\cdot)}\!\!\int_\Omega{} \mu _\lambda \mu_\lambda (\tau) \dx \d \tau \notag \\
	& \le \int_0^{(\cdot)} \| \mu_\lambda \|_{L^2(\Omega)} 
	\bigl\| \mu_\lambda (\tau) \bigr\|_{L^2(\Omega)} \d \tau 
	 \le \| \mu_\lambda \|_{L^2(\Omega)} \int_0^{(\cdot)} \bigl\| \mu_\lambda (\tau) \bigr\|_{L^2(\Omega)} \d \tau. 
	\label{add3}
\end{align}
Here, we substitute \eqref{add2} and \eqref{add3} into \eqref{add1}, and 
apply the Miranville--Zelik trick \cite[Appendix, Prop.~A.1]{MZ04} (see also \cite{GMS09} for a complete proof) to find some constants $\delta_1>0$, $C_1>0$ satisfying 
\begin{align*}
	& \delta_1 \int_\Omega \bigl| \beta _\lambda (u_\lambda) \bigr| \dx -C_1|\Omega| \le \int_\Omega \beta_\lambda(u_\lambda) (u_\lambda-m_0) \dx \\
	& \le \int_\Omega \beta_\lambda(u_\lambda) (u_\lambda-m_0) \dx + \int_\Omega |\nabla u_\lambda|^2 \dx\\
	& \le \frac{1}{K} \int_\Gamma (\phi_\lambda- u_\lambda)(u_\lambda-m_0) \dg 
	+ \bigl\| f- \pi(u_\lambda)-\lambda \partial _t u_\lambda \bigr\|_{L^2(\Omega)} 
	\| u_\lambda-m_0 \|_{L^2(\Omega)}
	\notag \\
	& \quad  {} 
	+ C_{\rm P} \| \nabla \mu_\lambda \|_{L^2(\Omega)} \| u_\lambda -m_0 \|_{L^2(\Omega)}  
		+ \lambda C_{\rm P} \| \nabla \mu_\lambda \|_{L^2(\Omega)} 
	\int_0^{(\cdot)}
	\bigl\| \mu_\lambda (\tau) \bigr\|_{L^2(\Omega)} \d \tau \notag \\
	& \quad {}
	+ \lambda \| \mu_\lambda \|_{L^2(\Omega)} \int_0^{(\cdot)} \bigl\| \mu_\lambda (\tau) \bigr\|_{L^2(\Omega)} \d \tau
\end{align*}
a.e.\ on $(0,T)$. At the same time, picking some $m_0' \in  {\rm int} D(\beta_\Gamma)$ we multiply \eqref{dbclam} by 
$\phi_\lambda-m_0'$ in order to deduce that 
\begin{align*}
	& \delta_1 \int_\Gamma \bigl| \beta _{\Gamma,\lambda} (\phi_\lambda) \bigr| \dg -C_1|\Gamma| \\
	& \le \int_\Gamma \beta_{\Gamma,\lambda}(\phi_\lambda) (\phi_\lambda-m_0') \dg + 
	\int_\Gamma |\nabla_\Gamma \phi_\lambda |^2 \dg \\
	& \le - \frac{1}{K} \int_\Gamma (\phi_\lambda- u_\lambda)(\phi_\lambda-m_0') \dg 
	+ \bigl\| f_\Gamma -\pi_\Gamma (\phi_\lambda ) -\partial _t \phi_\lambda \bigr\|_{L^2(\Gamma)}
	\| \phi_\lambda - m_0' \|_{L^2(\Gamma)}
\end{align*}
a.e.\ on $(0,T)$. We sum the last inequalities to deduce that
\begin{align*}
	& \delta_1 \int_\Omega \bigl| \beta _\lambda \bigl(u_\lambda(t) \bigr) \bigr| \dx
	+ \delta_1 \int_\Gamma \bigl| \beta _{\Gamma,\lambda} \bigl(\phi_\lambda(t) \bigr) \bigr| \dg 
	+ \frac{1}{K} \int_\Gamma\bigl | u_\lambda(t)-\phi_\lambda(t) \bigr|^2 \dg \\
	 & \le C_1\bigl( |\Omega| + |\Gamma| \bigr)
	 + \frac{1}{K} \int_\Gamma \bigl(u_\lambda(t)-\phi_\lambda(t) \bigr) (m_0-m_0') \dg \notag \\
	& \quad {}  + \bigl\| f(t)- \pi\bigl( u_\lambda (t) \bigr)-\lambda \partial _t u_\lambda(t) \bigr\|_{L^2(\Omega)} 
	\bigl\| u_\lambda(t) -m_0 \bigr\|_{L^2(\Omega)} \notag \\
	&\quad {} + \bigl\| f_\Gamma  (t)-\pi_\Gamma \bigl( {\phi_\lambda} (t) \bigr) -\partial _t \phi_\lambda  (t)
	\bigr\|_{L^2(\Gamma)} \bigl\| {\phi_\lambda} (t) - m_0' \bigr\|_{L^2(\Gamma)}
	\notag \\
	& \quad {} + C_{\rm P} \bigl\| \nabla \mu_\lambda (t) \bigr\|_{L^2(\Omega)} 
	\bigl\| u_\lambda(t) -m_0 \bigr\|_{L^2(\Omega)} 
	+ \lambda C_{\rm P} \bigl\| \nabla \mu_\lambda(t) \bigr\|_{L^2(\Omega)} 
	 \int_0^t \bigl\| \mu_\lambda (\tau) \bigr\|_{L^2(\Omega)} \d \tau
	 \notag \\
	 & \quad {}
	 + \lambda \bigl\| \mu_\lambda(t) \bigr\|_{L^2(\Omega)} 
	 \int_0^t \bigl\| \mu_\lambda (\tau) \bigr\|_{L^2(\Omega)} \d \tau\\
	 & \le C_1\bigl( |\Omega| + |\Gamma| \bigr)
	 + \frac{1}{2K} \int_\Gamma \bigl| u_\lambda(t)-\phi_\lambda(t) \bigr|^2 \dg 
	 +  \frac{1}{2K}|m_0-m_0'|^2 |\Gamma| \notag \\
	& \quad {}  
	+ \bigl\| f(t)- \pi\bigl( u_\lambda (t) \bigr)-\lambda \partial _t u_\lambda(t) \bigr\|_{L^2(\Omega)} 
	\| u_\lambda-m_0 \|_{L^\infty(0,T;L^2(\Omega))} \notag \\
	& \quad {} + \bigl\| f_\Gamma  (t)-\pi_\Gamma \bigl( {\phi_\lambda} (t) \bigr) -\partial _t \phi_\lambda  (t)
	\bigr\|_{L^2(\Gamma)} \| {\phi_\lambda} - m_0' \|_{L^\infty(0,T;L^2(\Gamma))}
	\notag \\
	& \quad {} + C_{\rm P} 
	\bigl\| \nabla \mu_\lambda (t) \bigr\|_{L^2(\Omega)} \biggl(
	\| u_\lambda-m_0 \|_{L^\infty(0,T;L^2(\Omega))} 
	+ \lambda \int_0^t \bigl\| \mu_\lambda (\tau) \bigr\|_{L^2(\Omega)} \d \tau \biggr)
	 \notag \\
	 & \quad {}
	 + \lambda \bigl\| \mu_\lambda(t) \bigr\|_{L^2(\Omega)} 
	 \int_0^t \bigl\| \mu_\lambda (\tau) \bigr\|_{L^2(\Omega)} \d \tau
\end{align*}
for a.a.\ $t \in (0,T)$. 
Thus, the estimate \eqref{est1} implies that there exists a positive constant 
$M_3'$ independent of $\lambda \in (0,1]$ and 
$K>0$ such that
\begin{equation*}
\begin{aligned}
&	\int_0^T \bigl\| \beta _\lambda (u_\lambda ) \bigr\|_{L^1(\Omega)}^2 \dt  + \int_0^T \bigl\| \beta _{\Gamma,\lambda} (\phi_\lambda ) \bigr\|_{L^1(\Gamma)}^2 \dt \\
& 
	 \le M_3' \left( 1 
	 +  \frac{1}{K^2}|m_0-m_0'|^4 
	+ \left(1+\frac{1}{K}\|u_0 -\phi_0\|_{L^2(\Gamma)}^2 \right)^{\!\! 2} \right).
    \end{aligned}
\end{equation*}
Indeed, please note that we could square both sides, integrate over $[0,T]$ and use \eqref{est1} to deal with the last terms above. For example, let us detail the treatment of
\begin{align*}
	& \int_0^T 
	\bigl\| f(t)- \pi\bigl( u_\lambda (t) \bigr)-\lambda \partial _t u_\lambda(t) \bigr\|_{L^2(\Omega)}^2
	\| u_\lambda-m_0 \|_{L^\infty(0,T;L^2(\Omega))}^2 \dt \notag \\
	& \le 
	\left\{ 
	3 \int _0^T \bigl\|f(t) \bigr\|^2_{L^2(\Omega)} \dt
	+3 \int _0^T \bigl\|\pi \bigl( u_\lambda(t) \bigr) \bigr\|^2_{L^2(\Omega)} \dt
	+ 3 \lambda^2 \int _0^T \bigl\|\partial _t u_\lambda(t) \bigr\|^2_{L^2(\Omega)} \dt
	\right\}
	\notag \\
	& \quad {} \times 
	\left\{ 2M_1^2 
	\left(1+\frac{1}{\sqrt{K}}\|u_0 -\phi_0\|_{L^2(\Gamma)} \right)^{\!\! 2}+ 2 m_0^2|\Omega|
	\right\}  \notag \\
	& \le M_3' \left(1+\frac{1}{K}\|u_0 -\phi_0\|_{L^2(\Gamma)}^2 \right)^{\!\! 2}.
\end{align*}
Also, we find that
\begin{align*}
    & \int_0^T
    \biggl( C_{\rm P}^2 \bigl\| \nabla \mu_\lambda (t) \bigr\|_{L^2(\Omega)}^2 + \bigl\| \mu_\lambda(t) \bigr\|_{L^2(\Omega)}^2 \biggr)  \lambda^2 \biggr ( \int_0^t \bigl\| \mu_\lambda (\tau) \bigr\|_{L^2(\Omega)} \d \tau \biggr)^2 \dt  \\
    & \le \int_0^T
    \biggl( \lambda C_{\rm P}^2 \bigl\| \nabla \mu_\lambda (t) \bigr\|_{L^2(\Omega)}^2  + \lambda \bigl\| \mu_\lambda(t) \bigr\|_{L^2(\Omega)}^2 \biggr) T \biggr ( \int_0^T \lambda \bigl\| \mu_\lambda (\tau) \bigr\|_{L^2(\Omega)}^2 \d \tau \biggr) \dt \\
    & \le M_3' \left(1+\frac{1}{K}\|u_0 -\phi_0\|_{L^2(\Gamma)}^2 \right)^{\!\! 2}.
\end{align*}
Namely, we can conclude that \eqref{est3} holds. 
\end{proof}

\begin{lemma}\label{L3} 
There exist positive constants $M_4$ and $M_5$, independent of $\lambda \in (0,1]$ and 
$K>0$, such that	
\begin{gather}
	\bigl\| m (\mu_\lambda) \bigr\|_{L^2(0,T)}
	\le M_4 \left(  1 + \frac{1}{K} |m_0-m_0'|^2 + \frac{1}{K} \|u_0 - \phi_0 \|_{L^2(\Gamma)}^2 \right),
	\label{est4}
	\\
	\| \mu_\lambda \|_{L^2(0,T;H^1(\Omega))} 
	\le M_5\left(  1 + \frac{1}{K} |m_0-m_0'|^2 +  \frac{1}{K} \|u_0 - \phi_0 \|_{L^2(\Gamma)}^2  \right).
	\label{est5}
\end{gather}
\end{lemma}

\begin{proof}
By integrating \eqref{ch2lam} over $\Omega$ and by comparison of terms from \eqref{rbclam} and \eqref{dbclam}, we have that
\begin{align*}
	\left| \int_\Omega \mu_\lambda \dx \right|
	& \le \int_\Omega \bigl| 
	\lambda \partial _t u_\lambda + \beta_\lambda(u_\lambda) +\pi(u_\lambda) - f \bigr| \dx \notag \\
	&\quad {} +\int_\Gamma \bigl|\partial_t \phi_\lambda + \beta_{\Gamma,\lambda}(\phi_\lambda) +\pi_\Gamma (\phi_\lambda) -f_\Gamma \bigr| \dg \\
	& \le \lambda \| 
	\partial_t u_\lambda \|_{L^2(\Omega)}|\Omega|^{1/2}
	+\int_\Omega \bigl| 
	\beta_\lambda(u_\lambda)\bigr| \dx
	+  \int_\Omega \bigl| \pi(u_\lambda) \bigr| \dx
	+ \| f \|_{L^2(\Omega)}|\Omega|^{1/2} \\
	& \quad {} 
	+ 
	\|\partial_t \phi_\lambda \|_{L^2(\Gamma)} |\Gamma|^{1/2}
	+ \int_\Gamma \bigl|\beta_{\Gamma,\lambda}(\phi_\lambda)\bigr| \dg
	+ \int_\Gamma \bigl|\pi_\Gamma (\phi_\lambda) \big| \dg
	+ \| f_\Gamma \|_{L^2(\Gamma)}|\Gamma|^{1/2}.
\end{align*}
Then, \eqref{est1} and \eqref{est3} imply
\eqref{est4}. 
Next, we can apply the Poincar\'e--Wirtinger inequality 
\begin{equation}
\label{pier3}
	\| \mu_\lambda \|_{H^1(\Omega)} \le 
	C_{{\rm PW}} \bigl( \| \nabla \mu_\lambda \|_{L^2(\Omega)} + \bigl|m(\mu_\lambda) \bigr| \bigr)
\end{equation}
a.e.\ in $(0,T)$ to deduce \eqref{est5}.  
\end{proof}
\smallskip

Hereafter, we use the notation $C(K,m_0')$ to denote a positive constant of the form
\begin{equation*}
	C(K,m_0') := M \left(  1 + \frac{1}{K} |m_0-m_0'|^2 + \frac{1}{K} \|u_0 - \phi_0 \|_{L^2(\Gamma)}^2 \right)
\end{equation*}
for $K>0$ and $m_0' \in {\rm int} D(\beta_\Gamma)$. 
In this definition, the generic constant $M>0$ stands for a sufficiently large constant independent of 
$\lambda \in (0,1]$ and $K>0$. While $M$ is always independent of $\lambda$ and $K$, 
its specific value may vary depending on the context, that is, from line to line. 
We also use the notation $C(K)$ for a positive constant as   
\begin{equation*}
	C(K) := M \left( 1+ \frac{1}{\sqrt{K}} \right) 
	\left(  1 + \frac{1}{\sqrt{K}}\|u_0 - \phi_0 \|_{L^2(\Gamma)} \right)
\end{equation*}
with constant $M>0$ large enough. Note that $C(K)$ still depends on $K$ even though we 
assume $u_0=\phi_0$ a.e.\ on $\Gamma$. Therefore, concerning asymptotic results, the next 
two lemmas will be useful only for the 
asymptotics as $K \to + \infty$. Of course, for the case when $K \to 0$, the following estimates will not be 
of help, however in that situation we will make additional assumptions, namely {\rm (A7)} and {\rm (A8)}, to obtain global bounds. 

\begin{lemma}\label{L4} 
For each $\lambda \in (0,1]$ and $K>0$ there holds
\begin{gather}
	\bigl\| \beta_{\Gamma,\lambda} ( \phi_\lambda ) \bigr\|_{L^2(0,T;L^2(\Gamma))}
	\le C(K),
	\label{est6} \\
	\| \Delta_\Gamma \phi_\lambda \|_{L^2(0,T;L^2(\Gamma))} 
	\le C(K),
	\label{est7} \\
	\| \phi_\lambda \|_{L^2(0,T;H^2(\Gamma))} 
	\le C(K).
	\label{est8}
\end{gather}
The last estimate \eqref{est8} implies the uniform boundedness of $\{ \phi_\lambda \}$ in $L^2(0,T;L^\infty(\Gamma))$.  
\end{lemma}

\begin{proof}
Multiplying \eqref{dbclam} by $\beta_{\Gamma,\lambda}(\phi_\lambda)$ and integrating it over $\Gamma$, 
by the Young inequality  we deduce that
\begin{align*}
	& \int_\Gamma \beta_{\Gamma,\lambda}' (\phi_\lambda) |\nabla_\Gamma \phi_\lambda|^2 \dg 
	+ \bigl\| \beta_{\Gamma,\lambda}(\phi_\lambda) \bigr\|_{L^2(\Gamma)}^2 \notag \\
	& \le \bigl\| f_\Gamma - \pi_\Gamma (\phi_\lambda)-\partial _t \phi_\lambda \bigr\|_{L^2(\Gamma)}^2
	+	
	\frac{1}{K^2} \| u_\lambda-\phi_\lambda \|_{L^2(\Gamma)}^2
	+
	\frac{1}{2} \bigl\| \beta_{\Gamma,\lambda}(\phi_\lambda) \bigr\|_{L^2(\Gamma)}^2,
\end{align*}
which leads to
\begin{align*}
	& \bigl\| \beta_{\Gamma,\lambda}(\phi_\lambda) \bigr\|_{L^2(0,T;L^2(\Gamma))} \\
	& \le M_6
	\left(1
	+ \frac{1}{\sqrt{K}} \| u_0-\phi_0 \|_{L^2(\Gamma)} \right)
	+ \frac{1}{\sqrt{K}}\, M_1
	\left(1
	+ \frac{1}{\sqrt{K}} \| u_0-\phi_0 \|_{L^2(\Gamma)} \right)
\end{align*}
for some positive constant $M_6$, 
from the monotonicity of $\beta_{\Gamma,\lambda}$. This implies \eqref{est6}. 
\smallskip 

Next, by comparing terms in the equation \eqref{dbclam}, we see that there exists a constant $M_7>0$ such that  
\begin{align*}
	& \| \Delta_\Gamma \phi_\lambda \|_{L^2(0,T;L^2(\Gamma))} \notag \\
	& \le \| \partial _t \phi_\lambda \|_{L^2(0,T;L^2(\Gamma))} 
	+ \bigl\| \beta_{\Gamma,\lambda}(\phi_\lambda) \bigr\|_{L^2(0,T;L^2(\Gamma))} 
	+ \bigl\| \pi_\Gamma (\phi_\lambda) \bigr\|_{L^2(0,T;L^2(\Gamma))} 
	\notag \\
	& \quad {} + \frac{1}{K} 
	\| \phi_\lambda - u_\lambda \|_{L^2(0,T;L^2(\Gamma))} 
	+\| f_\Gamma \|_{L^2(0,T;L^2(\Gamma))} \notag \\
	& \le M_1
	\left(1
	+ \frac{1}{\sqrt{K}} \| u_0-\phi_0 \|_{L^2(\Gamma)} \right) 
	+ \bigl\| \beta_{\Gamma,\lambda}(\phi_\lambda) \bigr\|_{L^2(0,T;L^2(\Gamma))} 
	\notag \\
	& \quad {}+ M_7 \left(1
	+ \frac{1}{\sqrt{K}} \| u_0-\phi_0 \|_{L^2(\Gamma)} \right) 
	\notag \\
	& \quad {}
	+ \frac{1}{\sqrt{K}}\, M_1
	\left(1
	+ \frac{1}{\sqrt{K}} \| u_0-\phi_0 \|_{L^2(\Gamma)} \right) 
	+ \| f_\Gamma \|_{L^2(0,T;L^2(\Gamma))} \notag \\
	& \le C(K).
\end{align*}
This yields the estimate \eqref{est7}, and by elliptic regularity also \eqref{est8} follows.
\end{proof}

\begin{lemma}\label{L5} 
For each $\lambda \in (0,1]$ and $K>0$ there holds
\begin{gather}
	\bigl\| \beta_{\lambda}( u_\lambda) \bigr\|_{L^2(0,T;L^2(\Omega))}
	\le C(K) + C(K,m_0'),
	\label{est9} \\
	\| \Delta u_\lambda \|_{L^2(0,T;L^2(\Omega))} 
	\le C(K) + C(K,m_0'),
	\label{est10} \\
	\| \partial _{\boldsymbol{\nu}} u_\lambda \|_{L^2(0,T;H^{1/2}(\Gamma))} 
	\le \frac{1}{\sqrt K}C(K),
	\label{est11} \\
	\| u_\lambda \|_{L^2(0,T;H^2(\Omega))} 
	\le C(K,m_0') + \left( 1+ \frac{1}{\sqrt K} \right) C(K).
	\label{est12}
\end{gather}  
\end{lemma}

\begin{proof}
Multiplying \eqref{ch2lam} by $\beta_\lambda(u_\lambda)$ and integrating it over $\Omega$, thanks to \eqref{rbclam} and to the subdifferential property for $\beta_\lambda$ we deduce that
\begin{align}
	& \int_\Omega \beta_\lambda' (u_\lambda) |\nabla u_\lambda|^2 \dx 
	+ \int_\Omega \bigl| \beta_\lambda(u_\lambda) \bigr|^2 \dx 
	\notag \\
	& \le \frac{1}{K} \int_\Gamma (\phi_\lambda-u_\lambda) \beta_\lambda (u_\lambda) \dg 
	+ \bigl\| \mu_\lambda -\lambda \partial_t u_\lambda -\pi (u_\lambda) + f \bigr\|_{L^2(\Omega)} 
	\bigl\| \beta_\lambda (u_\lambda) \bigr\|_{L^2(\Omega)} 
	\notag \\
	& \le \frac{1}{K} \int_\Gamma \hat{\beta}_\lambda (\phi_\lambda) \dg 
	-\frac{1}{K} \int_\Gamma \hat{\beta}_\lambda (u_\lambda) \dg 
	+ \frac{1}{2} \bigl\| \mu_\lambda -\lambda \partial_t u_\lambda -\pi (u_\lambda) + f \bigr\|_{L^2(\Omega)} ^2
	\notag \\
	& \quad {}
	+ \frac{1}{2} \bigl\| \beta_\lambda (u_\lambda) \bigr\|_{L^2(\Omega)}^2. 
	\label{trick}
\end{align}
Now, we recall the growth condition \eqref{gc} of the assumption {\rm (A5)} and observe that, for~$r,s \in \mathbb{R}$,
\begin{align*}
	\frac{1}{2\lambda} |r-s|^2 + \hat{\beta}(s) 
	& \le \frac{1}{2\lambda}|r-s|^2 + c_0\bigl( 1+|s|^2+ \hat{\beta}_\Gamma(s) \bigr) \\
	& \le \frac{1}{2\lambda}|r-s|^2 + c_0+2c_0|s-r|^2+2c_0|r|^2+ c_0 \hat{\beta}_\Gamma(s)\\
	& \le (1+4c_0)\frac{1}{2\lambda}|r-s|^2 + c_0+2c_0|r|^2+c_0 \hat{\beta}_\Gamma(s) \\
	& \le (1+4c_0)\left( 1+ |r|^2 + \frac{1}{2\lambda} |r-s|^2 + \hat{\beta}_\Gamma(s) \right),
\end{align*}
whence we can take the minimum with respect to $s$ in both sides finding that 
\begin{equation*}
	\hat{\beta}_\lambda(r) \le (1+4c_0 )\bigl( 1+ |r|^2 + \hat{\beta}_{\Gamma, \lambda}(r) \bigr)
	\quad {\rm for~} r \in \mathbb{R}. 
\end{equation*}
Thus, integrating \eqref{trick} over $(0,T)$ and exploiting \eqref{est1}, \eqref{est2}, and \eqref{est5}, we obtain 
\begin{align*}
	& \frac{1}{2}\int_0^T \bigl| \beta_\lambda(u_\lambda) \bigr|_{L^2(\Omega)}^2 \dt
	\notag \\
	& \le \frac{1}{K} \int_0^T \! \! \int_\Gamma \hat{\beta}_\lambda (\phi_\lambda) \dg \dt
	+ \frac{1}{2} \int_0^T \bigl\| \mu_\lambda -\lambda \partial_t u_\lambda -\pi (u_\lambda)
	+ f \bigr\|_{L^2(\Omega)} ^2  \dt \\
	& \le \frac{1}{K} (1+4c_0) \int_0^T \! \! \int_\Gamma \left( 1+ |\phi_\lambda|^2 + \hat{\beta}_{\Gamma, \lambda}
	\bigl( \phi_\lambda \bigr) \right)\dg \dt 
	+ 2 \int_0^T \| \mu_\lambda \|_{L^2(\Omega)} ^2 \dt \notag \\
	& \quad {}
	+ 2 \lambda^2 \int_0^T \| \partial_t u_\lambda \|_{L^2(\Omega)} ^2 \dt 
	+ 2 \int_0^T \bigl\| \pi (u_\lambda) \bigr\|_{L^2(\Omega)} ^2 \dt 
	+ 2 \int_0^T \|f\|_{L^2(\Omega)} ^2 \dt \notag \\
	& \le M_8 \frac{1}{K} \left(1
	+ \frac{1}{\sqrt{K}} \| u_0-\phi_0 \|_{L^2(\Gamma)} \right)^{\!\! 2} 
	+ C(K,m_0')^2+ M_8 \left(1
	+ \frac{1}{\sqrt{K}} \| u_0-\phi_0 \|_{L^2(\Gamma)} \right)^{\!\! 2},
\end{align*}
for some constant $M_8>0$. Hence, \eqref{est9} follows.
Moreover, from the comparison in the equation \eqref{ch2lam}, owing to \eqref{est5}, \eqref{est9}, and \eqref{est1} we infer that
\begin{align*}
	& \| \Delta u_\lambda \|_{L^2(0,T;L^2(\Omega))} \\
	& \le  \| \mu_\lambda \|_{L^2(0,T;L^2(\Omega))}
	   + \bigl\| \beta_\lambda(u_\lambda) \bigr\|_{L^2(0,T;L^2(\Omega))}
	    + \bigl\| f  -\lambda \partial_t u_\lambda -\pi (u_\lambda) \bigr\|_{L^2(0,T;L^2(\Omega))} \\
	   & \le C(K,m_0') + C(K), 
\end{align*}
which is \eqref{est10}. 
On the other hand, from the boundary condition \eqref{rbclam} and estimate~\eqref{est1} it follows that
\begin{align*}
	\| \partial _{\boldsymbol{\nu}} u_\lambda \|_{L^2(0,T;H^{1/2}(\Gamma))} 
	& \le \frac{1}{K} \| \phi_\lambda \|_{L^2(0,T;H^{1/2}(\Gamma))} + \frac{1}{K}\| u_\lambda \|_{L^2(0,T;H^{1/2}(\Gamma))} \notag \\
	& \le \frac{1}{K} \sqrt{T} \bigl( 1+ C_{\rm tr}\bigr) M'_8 \left(  1 + \frac{1}{\sqrt{K}}\|u_0 - \phi_0 \|_{L^2(\Gamma)} \right) ,
\end{align*}
with $M'_8 >0 $ sufficiently large and where $C_{\rm tr} >0$ is a constant acting in the trace inequality
$\| z \|_{H^{1/2}(\Gamma)} \le C_{\rm tr} \| z \|_{H^1(\Omega)}$ for all $z \in H^1(\Omega)$. 
Thus, the estimate \eqref{est11} holds. 
Finally, applying the elliptic estimate \eqref{eli} 
leads to \eqref{est12}.
\end{proof}

\subsection{Proof of Theorem 2.1.} Now, keeping 
$K>0$ fixed, we aim to pass to the limit as $\lambda \to 0$. More precisely, 
by virtue of the uniform estimates obtained above we can prove Theorem~\ref{lambda} as follows. 
\smallskip

\begin{proof} From the uniform estimates obtained in Lemmas~\ref{L1}, \ref{L2}, \ref{L3}, \ref{L4}, and \ref{L5}, we see that  
there exist a subsequence $\{ \lambda_n\}$, with $\lambda_n \to 0$ as $n \to +\infty$, and 
a quintuplet $(u, \xi, \mu, \phi, \psi)$ satisfying \eqref{clu}--\eqref{clpsi} and such that 
\begin{align*} 
	u_{\lambda_n} \to u 
	& \quad {\rm weakly~star~in~} \\ &\quad\quad
	H^1 \bigl( 0,T;H^1(\Omega)' \bigr) \cap L^\infty \bigl( 0,T ; H^1(\Omega) \bigr) \cap L^2 \bigl( 0,T ; H^2(\Omega) \bigr), \\
	\partial_{\boldsymbol{\nu}} u_{\lambda_n} \to 
	\partial_{\boldsymbol{\nu}} u 
	& \quad {\rm weakly~in~} 
	L^2 \bigl( 0,T ; L^2(\Gamma) \bigr), \\
	\lambda_n u_{\lambda_n} \to 0 
	& \quad {\rm strongly~in~}
	H^1\bigl( 0,T;L^2(\Omega)\bigr), \\
	\beta_{\lambda_n}(u_{\lambda_n}) \to \xi 
	& \quad {\rm weakly~in~} 
	L^2 \bigl( 0,T ; L^2(\Omega) \bigr), \\
	\mu_{\lambda_n} \to \mu 
	& \quad {\rm weakly~in~} 
	L^2 \bigl( 0,T ; H^1(\Omega) \bigr), \\
	\lambda_n \mu_{\lambda_n} \to 0 
	& \quad {\rm strongly~in~} L^2 \bigl( 0,T ; L^2(\Omega) \bigr), \\
	\phi_{\lambda_n} \to \phi 
	& \quad {\rm weakly~star~in~} \\ &\quad\quad
	H^1 \bigl( 0,T; L^2(\Gamma) \bigr) \cap L^\infty \bigl( 0,T ; H^1(\Gamma) \bigr) \cap L^2 \bigl( 0,T ; H^2(\Gamma) \bigr), \\
	\beta_{\Gamma,\lambda_n}(\phi_{\lambda_n}) \to \psi 
	& \quad {\rm weakly~in~} 
	L^2 \bigl( 0,T ; L^2(\Gamma) \bigr)
\end{align*}
as $n \to +\infty$. Moreover, applying the Aubin--Lions compactness theorems (see, e.g., \cite[Sect.~8, Cor.~4]{Sim87}) and invoking the Lipschitz continuity of $\pi$ and $\pi_\Gamma$, we obtain the strong convergence
\begin{align*} 
	u_{\lambda_n} \to u & \quad {\rm strongly~in~} C\bigl([0,T]; L^2(\Omega) \bigr) \cap 
	L^2 \bigl( 0,T ; H^1(\Omega) \bigr), \\
	\phi_{\lambda_n} \to \phi & \quad {\rm strongly~in~} C\bigl([0,T]; L^2(\Gamma) \bigr) \cap 
	L^2 \bigl( 0,T ; H^1(\Gamma) \bigr), \\
	\pi(u_{\lambda_n}) \to \pi(u) & \quad {\rm strongly~in~} C\bigl([0,T]; L^2(\Omega) \bigr), \\
	\pi_\Gamma(\phi_{\lambda_n}) \to \pi_\Gamma (\psi) & \quad {\rm strongly~in~} C\bigl([0,T]; L^2(\Gamma) \bigr)
\end{align*}
as $n \to +\infty$. Then, from the demi-closedness of the maximal monotone operators induced by $\beta$ and $\beta_\Gamma$, we can easily deduce that (cf., e.g., \cite[Prop.~2.2, p.~38]{Bar10})
\begin{equation*}
	\xi \in \beta(u) \quad {\rm a.e.\ in~} Q, \quad 
	\psi \in \beta_\Gamma (\phi) \quad {\rm a.e.\ on~} \Sigma.
\end{equation*}
Thus, we can pass to the limit in \eqref{ch2lam}, \eqref{dbclam}, \eqref{rbclam}, \eqref{ic1lam}, and \eqref{ic2lam} 
deducing \eqref{ch2}, \eqref{dbc}, \eqref{rbc}, \eqref{ic1}, and \eqref{ic2}, respectively.  
Moreover, about \eqref{ch1} we can pass to the limit in the weak formulation obtained from \eqref{ch1lam} with \eqref{nbclam}, that is,
\begin{equation*}
	\int_0^T \langle \partial _t u_{\lambda_n}, \eta \rangle_{H^1(\Omega)',H^1(\Omega)} \dt + 
	\int_0^T (\lambda_n \mu_{\lambda_n} , \eta )_{L^2(\Omega)} \dt 
	+ \int_0^T (\nabla \mu_{\lambda_n}, \nabla \eta )_{L^2(\Omega)} \dt 
	= 0
\end{equation*}
for all $\eta \in L^2 (0,T;H^1(\Omega))$. Thus, we complete the proof of the existence. 
\smallskip

Next, we prove the uniqueness properties. Before that, we recall an important tool often used in connection with Cahn--Hilliard systems.
Namely, for a given $\zeta\in H^1 (\Omega)'$, we consider the generalized Neumann problem of finding 
$z \in H^1 (\Omega)$ such that 
\begin{equation}
 \int_\Omega \nabla z \cdot \nabla v \dx
  = \langle \zeta , v \rangle_{H^1 (\Omega)', H^1 (\Omega)}
  \quad \text{\rm for~all~} v\in H^1 (\Omega).
  \label{neumann}
\end{equation}
Since $\Omega$ is connected and smooth, 
it is well known that the above problem admits solutions $z$ if and only if $\zeta$ has zero mean value. 
Hence, recalling the notation~\eqref{pier2}, we can introduce the following solution operator $\mathcal{N}$ by setting
\begin{equation*}
  \mathcal{N}:
\begin{array}{ccc}
	D(\mathcal{N}) := \bigl\{\zeta\in H^1 (\Omega)':\ m(\zeta)=0 \bigr\} &  \to & \bigl\{ z\in H^1 (\Omega):
  \ m(z)=0 \bigr\} \\
  \rotatebox{90}{\ensuremath{\in}} & & \rotatebox{90}{\ensuremath{\in}} \\
  \zeta & \mapsto & z, 
\end{array}
\end{equation*}
where $z$ uniquely solves \eqref{neumann} along with $m(z)=0$.
It turns out that $\mathcal{N}$ is an isomorphism between the above spaces and that the formula
\begin{equation}
  \|\zeta \|_*^2 := \int_\Omega \bigl| \nabla\mathcal{N} \bigl( \zeta-m(\zeta) \bigr)\bigr|^2 \dx + \bigl| m(\zeta) \bigr|^2
  \quad \hbox{for~all~} \zeta\in H^1 (\Omega)'
  \label{normstar}
\end{equation}
defines a Hilbert norm in $H^1 (\Omega)'$ that is equivalent to the standard norm of~$H^1 (\Omega)'$.
From the above properties, one can obtain the following identities:
\begin{gather}
	\int_\Omega \nabla \mathcal{N} \zeta \cdot \nabla v \dx
   = \langle \zeta , v \rangle_{H^1 (\Omega)', H^1 (\Omega)}
  \quad \text{for all } \zeta \in D(\mathcal{N}) \ \text{and} \ v\in H^1 (\Omega),
  \label{dadefN}\\
  \langle \zeta ,\mathcal{N}\zeta \rangle_{H^1 (\Omega)', H^1 (\Omega)}
  =\int_\Omega |\nabla\mathcal{N}\zeta|^2 \dx
  = \|\zeta \|_*^2 \quad \text{for all } \zeta \in D(\mathcal{N}),
  \label{pier12}
\end{gather}
if $\zeta \in D(\mathcal{N})\cap L^2(\Omega)$ then $\mathcal{N} \zeta \in H^2(\Omega)$ and 
\begin{gather}
\begin{cases}
  -\Delta\mathcal{N} \zeta 
  = \zeta \quad \text{a.e. in } \Omega, \\
  \partial_{\boldsymbol{\nu}} \mathcal{N} \zeta = 0 \quad \text{a.e. on } \Gamma. 
\end{cases}
  \label{laplaceN}
\end{gather}
We also point out that
\begin{equation}
 \bigl\langle \partial_t \zeta(t) ,\mathcal{N}\zeta(t) \bigr\rangle_{H^1 (\Omega)', H^1 (\Omega)}
  = \frac 12 \, \frac{\d}{\dt} \, \bigl \|\zeta(t) \bigr\|_*^2
  \quad \hbox{for a.a.\ $t\in(0,T)$},
  \label{propN} 
\end{equation}
which holds true for every $\zeta \in H^1(0,T;H^1(\Omega)') \cap L^2(0,T;H^1(\Omega))$ satisfying $m(\zeta) =0$ a.e.\ in~$(0,T)$.
\smallskip

Then, let now $(u^{(i)}, \xi^{(i)}, \mu^{(i)}, \phi^{(i)}, \psi^{(i)})$ be two sets of solutions $(i=1,2)$
to the problem stated in \eqref{clu}--\eqref{ic2} and put $\bar{u}:=u^{(2)}-u^{(1)}$, analogously defining $\bar{\xi}$, $\bar{\mu}$, $\bar{\phi}$, $\bar{\psi}$, respectively. 
Recalling \eqref{const}, note that $m(\bar{u})=m(u^{(2)})-m(u^{(1)})=m_0-m_0=0$ in $[0,T]$. 
Taking the difference between \eqref{ch1} written for $u^{(2)}$ and $u^{(1)}$ gives
\begin{equation}
	\langle \partial _t \bar{u}, y \rangle_{H^1(\Omega)', H^1(\Omega)} + 
	\int_\Omega \nabla \bar{\mu} \cdot \nabla y \dx = 0 \quad {\rm for~all~} y \in H^1(\Omega) ,
	\label{dual}
\end{equation}
a.e.\ on $(0,T)$. Now, we can  
take $y =\mathcal{N}\bar{u}$ as test function in \eqref{dual} and, due to \eqref{propN} and \eqref{dadefN}, deduce that
\begin{equation*}
	\frac{1}{2} \frac{\d}{\dt} \| \bar{u}\|_{*}^2 
	+ (\bar{\mu},\bar{u})_{L^2(\Omega)} = 0
	\quad {\rm a.e.\ in~} (0,T).
\end{equation*}
Next, testing the difference of \eqref{ch2} by  $\bar{u}$ and using \eqref{rbc}, we have 
\begin{equation*}
	 \|\nabla  \bar{u} \|_{L^2(\Omega)}^2 + 
	(\bar{\xi}, \bar{u})_{L^2(\Omega)} + \bigl( \pi(u^{(2)}) -\pi(u^{(1)}), \bar{u} )_{L^2(\Omega)}
	- \frac{1}{K} (\bar{\phi}-\bar{u}, \bar{u})_{L^2(\Gamma)} = (\bar{\mu},\bar{u})_{L^2(\Omega)} .  
\end{equation*}
Moreover, testing the difference of \eqref{dbc} by $\bar{\phi}$ we obtain
\begin{align*}
	& \frac{1}{2} \frac{\d}{\dt} \| \bar{\phi}\|_{L^2(\Gamma)}^2 
	 +  \int_\Gamma | \nabla _\Gamma \bar{\phi} |^2 \dg  
	 \notag \\
	 & \quad {}+ (\bar{\psi}, \bar{\phi})_{L^2(\Gamma)} + 
	 \bigl( \pi_\Gamma(\phi^{(2)}) -\pi_\Gamma(\phi^{(1)}), \bar{\phi} \bigr)_{L^2(\Gamma)}
	+ \frac{1}{K} (\bar{\phi}-\bar{u}, \bar{\phi})_{L^2(\Gamma)}=0.  
\end{align*}
Combining the last three equations with a cancellation of terms, thanks to the monotonicity of $\beta$ and $\beta_\Gamma$, we infer that
\begin{align}
	& \frac{1}{2} \frac{\d}{\dt} \| \bar{u}\|_{*}^2 
	+ \frac{1}{2} \frac{\d}{\dt} \| \bar{\phi}\|_{L^2(\Gamma)}^2 
	+ \|\nabla  \bar{u} \|_{L^2(\Omega)}^2
	+ \frac{1}{K} \| \bar{\phi}-\bar{u} \|_{L^2(\Gamma)}^2 
	\notag \\
	& \le C_{\rm L} \Bigl( \| \bar{u} \|_{L^2(\Omega)}^2 
	+ \| \bar{\phi}\|_{L^2(\Gamma)}^2 \Bigr) \notag \\
	& \le C_{\rm L} \Bigl( \delta \|\nabla  \bar{u} \|_{L^2(\Omega)}^2 + C_\delta \| \bar{u} \|_{*}^2
	+ \| \bar{\phi}\|_{L^2(\Gamma)}^2 \Bigr), \label{pier4}
\end{align}
a.e.\ in $(0,T)$, where $C_{\rm L}$ is a positive constant related to the Lipschitz constants of $\pi$ and $\pi_\Gamma$. 
Moreover, as $m(\bar{u})=0$ in $[0,T]$ here we could use the Ehrling--Lions lemma (see, e.g., \cite[Lemme 5.1, p.~58]{Lio69}) along with the Poincar\'e--Wirtinger inequality (cf.~\eqref{pier3}), so that for each $\delta>0$ there exists a positive constant $C_\delta$ such that
\begin{equation}
	\| \bar{u} \|_{L^2(\Omega)}^2 \le 
	\delta \|\nabla  \bar{u} \|_{L^2(\Omega)}^2 + C_\delta \| \bar{u} \|_{*}^2.
	\label{EL}
\end{equation}
Finally, we choose $\delta$ sufficiently small in \eqref{pier4}, e.g., $C_L \delta \leq 1/2$, then apply the Gronwall inequality 
with null initial conditions for $\bar{u}$ and $\bar{\phi}$, obtained from \eqref{ic1} and \eqref{ic2}.
\smallskip

Hence, we prove that $\bar{u} =0$ a.e.\ in $Q$ and $\bar{\phi} =0$ a.e.\ on $\Sigma$. This entails the uniqueness of the components $u$ and $\phi$ solving \eqref{clu}--\eqref{ic2}. 
Moreover, by the comparison in the equation \eqref{dbc}, $\psi$ is unique. 
If $\beta$ is single-valued, then $\xi = \beta(u)$ is also unique, and by the comparison in the equation \eqref{ch2}, $\mu$ is uniquely determined as well.
\end{proof}

\section{Asymptotic behavior related to zero and full permeability}
\label{Kzeroandinfty}
\setcounter{equation}{0} 

In this section, we discuss the asymptotic behavior related to $K \to \infty$ and $K \to 0$.

\subsection{Asymptotic behavior related to zero permeability}
As seen in the previous section, more precisely from the uniform estimates obtained in Lemmas~\ref{L1}, \ref{L2}, \ref{L3}, \ref{L4}, and~\ref{L5}, 
it is possible to study the limiting procedure as $K \to +\infty$. 
This corresponds to the case where the permeability $\alpha := 1/K$ tends to $0$; in other words, the system between the bulk and the boundary becomes completely decoupled. 
We can rigorously prove this situation as follows.

\begin{theorem}
\label{a0} 
Assume {\rm (A1)}--{\rm (A6)}. Let $(u_\alpha,\xi_\alpha, \mu_\alpha,\phi_\alpha, \psi_\alpha)$ 
be the solution to \eqref{clu}--\eqref{ic2} obtained in {\rm Theorem~\ref{lambda}}. 
Then, there exists a unique function 
\begin{equation*}
	u \in H^1 \bigl( 0,T;H^1(\Omega)' \bigr) \cap L^\infty \bigl( 0,T ; H^1(\Omega) \bigr) \cap L^2 \bigl( 0,T ; H^2(\Omega) \bigr)
\end{equation*}
such that
\begin{align}
	u_{\alpha} \to u & \quad \mbox{\it weakly~star~in~} 
	H^1 \bigl( 0,T;H^1(\Omega)' \bigr) \cap L^\infty \bigl( 0,T ; H^1(\Omega) \bigr) \cap L^2 \bigl( 0,T ; H^2(\Omega) \bigr) \notag \\ 
	& \quad \mbox{and strongly~in~} C\bigl( [0,T]; L^2(\Omega) \bigr) \cap L^2 \bigl(0,T;H^1(\Omega) \bigr)\quad \mbox{as~} \alpha \to 0 \label{pier5}
\end{align}
and there exist a subsequence $\{\alpha_n\}$ and functions 
\begin{equation*}
	\xi \in  L^2 \bigl( 0,T ; L^2(\Omega) \bigr), \quad 
	\mu  \in L^2 \bigl( 0,T ; H^1(\Omega) \bigr)
\end{equation*}
such that
\begin{align*}
	\xi_{\alpha_n} \to \xi & \quad \mbox{\it weakly~in~}  L^2 \bigl( 0,T ; L^2(\Omega) \bigr), \\ 
	\mu_{\alpha_n} \to \mu & \quad \mbox{\it weakly~in~}  L^2 \bigl( 0,T ; H^1(\Omega) \bigr)
	\quad \mbox{\it as~} n \to +\infty.
\end{align*}
Moreover, the triplet $(u,\xi, \mu)$
satisfies the following Cahn--Hilliard system in the bulk $Q$:
\begin{align*}
	\langle \partial _t u, y \rangle_{H^1(\Omega)', H^1(\Omega)} +\int_\Omega \nabla \mu \cdot \nabla y \dx = 0 \quad {\it for~all~} y \in H^1(\Omega), & \quad {\it a.e.\ on~}(0,T), \\
	-\Delta u + \xi + \pi (u) - f  = \mu, \quad \xi \in \beta(u) & \quad {\it a.e.\ in~} Q, \\
	\partial_{\boldsymbol{\nu}} u = 0 & \quad {\it a.e.\ on~} \Sigma, \\
	u(0) = u_0 & \quad {\it a.e.\ in~}\Omega. 
\end{align*}
Next, there exist a unique pair $(\phi, \psi)$ of functions 
\begin{gather*}
	\phi \in H^1 \bigl( 0,T;L^2(\Gamma) \bigr) \cap L^\infty \bigl( 0,T ; H^1(\Gamma) \bigr) \cap L^2 \bigl( 0,T ; H^2(\Gamma) \bigr), \\
	\psi \in  L^2 \bigl( 0,T ; L^2(\Gamma) \bigr) 
\end{gather*}
such that
\begin{align}
	\phi_\alpha \to \phi & \quad \mbox{\it weakly~star~in~} H^1 \bigl( 0,T;L^2(\Gamma) \bigr) \cap L^\infty \bigl( 0,T ; H^1(\Gamma) \bigr) \cap L^2 \bigl( 0,T ; H^2(\Gamma) \bigr) \notag\\
	& \quad \mbox{\it strongly~in~} C\bigl( [0,T]; L^2(\Gamma) \bigr) \cap L^2 \bigl(0,T;H^1(\Gamma) \bigr), 
	\label{pier6}\\
	\psi_\alpha \to \psi & \quad \mbox{\it weakly~in~} L^2 \bigl( 0,T ; L^2(\Gamma) \bigr) 
	\quad \mbox{\it as~} \alpha \to 0. \label{pier7} 
\end{align}
Moreover, the pair $(\phi, \psi)$ 
satisfies the following  Allen--Cahn equation on the boundary $\Sigma$: 
\begin{align*}
	\partial _t \phi - \Delta_\Gamma \phi + \psi + \pi_\Gamma (\phi) =f_\Gamma,\quad \psi \in \beta_{\Gamma} (\phi) & \quad {\it a.e.\ on~} \Sigma,
	\\
	\phi(0) = \phi_0 & \quad {\it a.e.\ on~} \Gamma. 
\end{align*}
Finally, the following rate of convergence holds as $\alpha \to 0$:
\begin{equation}
\label{pier8}
	\| u-u_\alpha \|_{C([0,T];H^1(\Omega)') \cap L^2(0,T;H^1(\Omega))} + 
	\| \phi-\phi_\alpha \|_{C([0,T];L^2(\Gamma)) \cap L^2(0,T;H^1(\Gamma))}
	= O(\alpha^{1/2}) .
\end{equation} 

\end{theorem}

\begin{proof}
The proof of this theorem is quite standard. From the uniform estimates obtained in the previous section, more precisely, the same kind of estimates hold with  
$u_\lambda$, $\mu_\lambda$, $\phi_\lambda$ replaced by $u_\alpha$, $\mu_\alpha$, $\phi_\alpha$, 
and with $\beta_\lambda(u_\lambda)$, $\beta_{\Gamma, \lambda}(\phi_\lambda)$ 
replaced by $\xi_\alpha$, $\psi_\alpha$, respectively. 
Thus, there exist a subsequence $\{\alpha_n\}$ and a quintuple $(u,\xi,\mu,\phi,\psi)$ such that 
the same kind of convergences as in the proof of Theorem~\ref{lambda} hold. 
Of course, the relations $\xi \in \beta(u)$ a.e.\ in $Q$ and $\psi \in \beta_\Gamma(\phi)$ a.e.\ on $\Sigma$ are a consequence of the demi-closedness of $\beta$ and $\beta_\Gamma$, together with $\xi_\alpha \in \beta(u_\alpha)$ a.e.\ in $Q$ and $\psi_\alpha \in \beta_\Gamma(\phi_\alpha)$ a.e.\ on $\Sigma$ (see \eqref{ch2} and \eqref{dbc}).
Especially, by virtue of \eqref{est1} and \eqref{est11}, we have that
\begin{align*}
	\partial _{\boldsymbol{\nu}} u_{\alpha_n} = \alpha_n(\phi_{\alpha_n}-u_{\alpha_n}) \to 0 & \quad  {\rm strongly~in~}L^\infty \bigl( 0,T;L^2(\Gamma) \bigr),\\
	\partial _{\boldsymbol{\nu}} u_{\alpha_n}  \to 0 & \quad 
	 {\rm strongly~in~}L^2 \bigl( 0,T; H^{1/2}(\Gamma) \bigr)
\end{align*} 
as $n \to +\infty$.   
Thus, taking the limit in \eqref{ch1}--\eqref{ic2} we conclude that 
$(u,\xi, \mu)$ solves the Cahn--Hilliard system and 
$(\phi, \psi)$ satisfies the Allen--Cahn equation. 
The proofs of uniqueness for $(\phi, \psi)$ in the Allen--Cahn equation, 
and for $u$ in the Cahn--Hilliard system are standard (quite similar to the proof of Theorem~\ref{lambda} above, therefore we omit them). Owing to uniqueness, it turns out that
the convergences \eqref{pier5}, \eqref{pier6}, \eqref{pier7} hold for the full sequence. Of course if the pair
$(\xi, \mu)$ is also unique, the convergence for 
$\xi_\alpha$ and $\mu_\alpha$ also hold for full family as $\alpha \to 0$. 
\smallskip

As a remark, the well-posedness of the limit problems can also be proven by alternative methods, since the bulk problem and the boundary problem are independent of each other and each constitutes an initial value problem for a well-known evolution differential system.
\smallskip

Finally, we prove the rate of convergence. 
We start by recalling \eqref{dual}, since the structure here is the same as in the proof of the uniqueness 
for Theorem~\ref{lambda}. First, we have 
$m(u(t)-u_\alpha(t))=0$ for all $t \in [0,T]$, therefore we can deduce --- from the difference between 
the weak formulation of Cahn--Hilliard system of $u$ and \eqref{ch1} --- that
\begin{equation}
	\frac{1}{2} \frac{\d}{\dt} \| u-u_\alpha \|_{*}^2 
	 +  ( \mu-\mu_\alpha, u-u_\alpha)_{L^2(\Omega)}=0
	\label{rate1}
\end{equation}
a.e.\ on $(0,T)$, by using the test function  $y =\mathcal{N} (u-u_\alpha)$. 
Moreover, from the difference between 
the equation of $\mu$ and \eqref{ch2} it follows that
\begin{align}
	 & \bigl\| \nabla(u-u_\alpha) \bigr\|_{L^2(\Omega)}^2 
	+ \alpha (\phi_\alpha-u_\alpha, u-u_\alpha)_{L^2(\Gamma)} + 
	(\xi-\xi_\alpha, u-u_\alpha)_{L^2(\Omega)}
	 \notag \\
	& \quad {} 
	+ \bigl( \pi(u) -\pi(u_\alpha),u-u_\alpha \bigr)_{L^2(\Omega)} = (\mu-\mu_\alpha,u-u_\alpha)_{L^2(\Omega)},
	\label{rate2}
\end{align}
where we should take care that $\partial_{\boldsymbol{\nu}} u =0$ and 
$\partial_{\boldsymbol{\nu}} u_\alpha =\alpha(\phi_\alpha-u_\alpha)$
a.e.\ on $\Sigma$. 
Moreover, testing the difference of the Allen--Cahn equation on the boundary of  
$\phi$
and \eqref{dbc} by $\phi-\phi_\alpha$, we have that
\begin{align}
	& \frac{1}{2} \frac{\d}{\dt} \| \phi- \phi_\alpha \|_{L^2(\Gamma)}^2 
	 +  \int_\Gamma \bigl| \nabla _\Gamma (\phi- \phi_\alpha) \bigr|^2 \dg  
	 + (\psi-\psi_\alpha, \phi- \phi_\alpha)_{L^2(\Gamma)}
	 \notag \\
	 & \quad {} + 
	 \bigl( \pi_\Gamma(\phi) -\pi_\Gamma(\phi_\alpha), \phi- \phi_\alpha \bigr)_{L^2(\Gamma)}
	- \alpha (\phi_\alpha-u_\alpha, \phi- \phi_\alpha)_{L^2(\Gamma)}=0. \
	\label{rate3}
\end{align}
We merge \eqref{rate1}, \eqref{rate2}, and \eqref{rate3} to infer that
\begin{align*}
	&\frac{1}{2} \frac{\d}{\dt} \|u-u_\alpha \|_{*}^2 
	+ \frac{1}{2} \frac{\d}{\dt} \| \phi-\phi_\alpha \|_{L^2(\Gamma)}^2 
	+ \bigl\| \nabla(u-u_\alpha) \bigr\|_{L^2(\Omega)}^2 
	\notag \\
	& \quad {} +
	\int_\Gamma \bigl| \nabla _\Gamma (\phi- \phi_\alpha) \bigr|^2 \dg  
	+ \alpha \|  \phi_\alpha - u_\alpha  \|_{L^2(\Gamma)}^2 
	\notag \\
	& \le C_{\rm L} \bigl( \| u-u_\alpha \|_{L^2(\Omega)}^2 + \| \phi- \phi_\alpha \|_{L^2(\Gamma)}^2 \bigr)
	+ \alpha(u_\alpha-\phi_\alpha,u-\phi)_{L^2(\Gamma)}
	\notag \\
	& \le C_{\rm L} \bigl( \delta  \bigl\| \nabla(u-u_\alpha) \bigr\|_{L^2(\Omega)}^2  + C_\delta \| u-u_\alpha \|_{*}^2 + \| \phi- \phi_\alpha \|_{L^2(\Gamma)}^2 \bigr)
	+ \frac{\alpha}{2} \|  \phi_\alpha - u_\alpha  \|_{L^2(\Gamma)}^2
	\notag \\
	& \quad {}
	+  \frac{\alpha}{2} \Bigl( C_{\rm tr}\| u \|_{L^\infty(0,T;H^1(\Omega))}^2 + \| \phi \|_{L^\infty(0,T;L^2(\Gamma))}^2 \Bigr)
\end{align*}
a.e.\ on $(0,T)$, where we used the monotonicity of $\beta$ and $\beta_\Gamma$ and the inequality~\eqref{EL} for $ u-u_\alpha$. Hence,
we can choose $ \delta = 1/(2C_L)$, then integrate over $(0,t)$, $t\in (0,T]$, with the help of the null initial conditions.
We obtain the final inequality 
\begin{align*}
	&\frac{1}{2} \bigl\| (u-u_\alpha)(t) \bigr\|_{*}^2 
	+ \frac{1}{2} \bigl\| (\phi-\phi_\alpha)(t) \bigr\|_{L^2(\Gamma)}^2 
	+ \frac12 \int_0^t\!\!\int_\Omega \bigl| \nabla(u-u_\alpha) \bigr|^2 \dx\ds
	\notag \\
	& \quad {} +
	\int_0^t\!\!\int_\Gamma \bigl| \nabla _\Gamma (\phi- \phi_\alpha) \bigr|^2 \dg\ds 
	+ \frac{\alpha}2 \int_0^t\!\!\int_\Gamma | \phi_\alpha - u_\alpha |^2  \dg\ds
	\notag \\
	& \le C \biggl( \int_0^t  \| u-u_\alpha \|_{*}^2\ds  + \int_0^t \| \phi- \phi_\alpha \|_{L^2(\Gamma)}^2\ds \biggr)
	\notag \\
	& \quad {}
	+  \frac{\alpha}{2} T \Bigl( C_{\rm tr}\| u \|_{L^\infty(0,T;H^1(\Omega))}^2 + \| \phi \|_{L^\infty(0,T;L^2(\Gamma))}^2 \Bigr)
\end{align*}
for some constant $C$ independent of $\alpha$.
At this point, we can apply the Gronwall lemma to conclude that \eqref{pier8} holds true.
\end{proof}

\subsection{Asymptotic behavior related to full permeability}
Formally, by taking the limit $K \to 0$ in the third type transmission condition \eqref{RBC}, one formally recovers the {D}irichlet trace condition $u = \phi $ on $\Sigma$. 
More precisely, this boundary condition is understood in the sense of traces, namely
\begin{equation*}
	\gamma \, u(t) = \phi(t) \quad \text{in } H^{1/2}(\Gamma), \quad \text{for a.a.\ } t \in (0,T).
\end{equation*}
We now proceed to rigorously analyze this asymptotic limit.
\smallskip

In order to discuss the limiting procedure as $K \to 0$, it is reasonable to assume that
\begin{enumerate}
	\item[(A7)] $u_0= \phi_0$ \ a.e.\ on $\Gamma$.
\end{enumerate}
Under the assumption {\rm (A4)} we have $u_0 \in H^1(\Omega)$, 
therefore the trace of $u_0$ makes sense, 
that is, using the trace operator 
$\gamma:H^1(\Omega) \to H^{1/2}(\Gamma)$ the above assumption 
can be rewritten rigorously as $\gamma \, u_0 =\phi_0$ in $H^{1/2} (\Gamma)$. 
\smallskip

Next, as suggested by the proof of Lemma~\ref{L3}, it is natural to take \( m_0' = m_0 \). To this end, we assume that
\begin{enumerate}
	\item[(A8)] $D(\beta_\Gamma)=D(\beta)$ and there exist constants $\rho_1 \ge 1$ and 
$c_1 >0$ such that
\begin{equation*}
	\frac{1}{\rho_1} \bigl| \beta_\Gamma ^\circ (r) \bigr|-c_1 
	\le \bigl| \beta ^\circ (r) \bigr|
	\le \rho_1 \bigl| \beta_\Gamma ^\circ (r) \bigr|+c_1
	\quad {\rm for~all~} r \in D(\beta_\Gamma)=D(\beta).
\end{equation*}
\end{enumerate}
Here the minimal section $\beta^\circ$ of multivalued graph $\beta$ is defined by 
$\beta^\circ(r):=\{ r^* \in \beta(r) : |r^*|=\min_{s \in \beta(r)} |s|\}$. 
$\beta^\circ_\Gamma$ is defined in the same way. 
Moreover, (A8) is the same as in the previous contribution \cite[cf.~Lemmas A.1 and A.2]{CF20}: in fact, 
under the assumption (A8) it turns out that the same inequalities hold at the level of Moreau--Yosida regularizations
\begin{equation}
	\frac{1}{\rho_1} \bigl| \beta_{\Gamma,\lambda} (r) \bigr|-c_1 
	\le \bigl| \beta_\lambda (r) \bigr|
	\le \rho_1 \bigl| \beta_{\Gamma,\lambda} (r) \bigr|+c_1
	\quad {\rm for~all~} r \in \mathbb{R},
	\label{NODEA} 
\end{equation}
for $\lambda \in (0,1]$. 
Hereafter, we additionally assume {\rm (A7)--(A8)}. 
\smallskip

We now state our main theorem. 
\begin{theorem}
\label{K0} 
Assume {\rm (A1)}--{\rm (A8)}. Let $(u_K, \xi_K, \mu_K,\phi_K, \psi_K)$ 
be the solution to \eqref{clu}--\eqref{ic2} obtained in {\rm Theorem~\ref{lambda}}. 
Then, there exist unique functions 
\begin{gather*}
	u \in H^1 \bigl( 0,T;H^1(\Omega)' \bigr) \cap L^\infty \bigl( 0,T ; H^1(\Omega) \bigr) \cap L^2 \bigl( 0,T ; H^2(\Omega) \bigr), \\
	\phi \in H^1 \bigl( 0,T;L^2(\Gamma) \bigr) \cap L^\infty \bigl( 0,T ; H^1(\Gamma) \bigr) \cap L^2 \bigl( 0,T ; H^2(\Gamma) \bigr), \\
	\psi \in L^2 \bigl( 0,T ; L^2(\Gamma) \bigr),
\end{gather*}
and the functions 
\begin{equation*}
	\xi \in  L^2 \bigl( 0,T ; L^2(\Omega) \bigr), \quad 
	\mu  \in L^2 \bigl( 0,T ; H^1(\Omega) \bigr)
\end{equation*}
such that the quintuple $(u,\xi, \mu, \phi, \psi)$
satisfies the following Cahn--Hilliard system in the bulk $Q$ with the 
dynamic boundary condition of Allen--Cahn type on the boundary $\Sigma$:
\begin{align*}
	\langle \partial _t u, y \rangle_{H^1(\Omega)', H^1(\Omega)} +\int_\Omega \nabla \mu \cdot \nabla y \dx = 0 \quad {\it for~all~} y \in H^1(\Omega), & \quad {\it a.e.\ on~}(0,T), \\
	-\Delta u + \xi + \pi (u) - f =\mu, \quad \xi \in \beta(u) & \quad {\it a.e.\ in~} Q, \\
	u = \phi & \quad {\it a.e.\ on~} \Sigma, \\
	\partial _t \phi  
	- \Delta_\Gamma \phi
	+ \psi + \pi_\Gamma (\phi)+\partial_{\boldsymbol{\nu}} u
	=f_\Gamma, \quad 
	\psi \in \beta_{\Gamma} (\phi) & \quad {\it a.e.\ on~} \Sigma,
	\\
	u(0) = u_0 & \quad {\it a.e.\ in~}\Omega,\\
	\phi(0) = \phi_0 & \quad {\it a.e.\ on~} \Gamma. 
\end{align*}
Moreover, it holds that
\begin{align}
	u_{K} \to u & \quad \mbox{\it weakly~star~in~} 
	H^1 \bigl( 0,T;H^1(\Omega)' \bigr) \cap L^\infty \bigl( 0,T ; H^1(\Omega) \bigr), 
	\label{conv1}\\ 
	& \quad \mbox{\it strongly~in~} C\bigl( [0,T]; L^2(\Omega) \bigr),
	\label{conv2}\\
	\Delta u_{K} \to \Delta u & \quad \mbox{\it weakly~in~} 
	L^2 \bigl( 0,T ; L^2(\Omega) \bigr), 
	\label{conv3}\\
	\partial_{\boldsymbol{\nu}} u_{K} \to \partial_{\boldsymbol{\nu}} u & \quad \mbox{\it weakly~in~} 
	L^2 \bigl( 0,T ; H^{-1/2}(\Gamma) \bigr), 
	\label{conv4}\\
	\phi_K \to \phi & \quad \mbox{\it weakly~star~in~} H^1 \bigl( 0,T;L^2(\Gamma) \bigr) \cap L^\infty \bigl( 0,T ; H^1(\Gamma) \bigr), 
	\label{conv5}\\
	& \quad \mbox{\it strongly~in~} C\bigl( [0,T]; L^2(\Gamma) \bigr), 
	\label{conv6}\\
	\Delta_\Gamma \phi_K \to \Delta_\Gamma \phi & \quad \mbox{\it weakly~in~} L^2 \bigl( 0,T;H^{-1/2}(\Gamma) \bigr), 
	\label{conv7}\\
	\psi_K \to \psi & \quad \mbox{\it weakly~in~} L^2 \bigl( 0,T ; L^2(\Gamma) \bigr) 
	\quad \mbox{\it as~} K \to 0,\label{conv8}
\end{align}
and there exists a subsequence $\{K_n\}$ 
such that
\begin{align}
	\xi_{K_n} \to \xi & \quad \mbox{\it weakly~in~}  L^2 \bigl( 0,T ; L^2(\Omega) \bigr), 
	\label{conv9}\\ 
	\mu_{K_n} \to \mu & \quad \mbox{\it weakly~in~}  L^2 \bigl( 0,T ; H^1(\Omega) \bigr)
	\quad \mbox{\it as~} n \to +\infty.
	\label{conv10}
\end{align}
Additionally, the following rate of convergence occurs as $K\to 0$:
\begin{gather}
	\| u-u_K \|_{C([0,T];H^1(\Omega)') \cap L^2(0,T;H^1(\Omega))}  
	 + 
	\| \phi-\phi_K \|_{C([0,T];L^2(\Gamma)) \cap L^2(0,T;H^1(\Gamma))}
	= O(K^{1/2}),
	\label{takeshi1} \\
     \| u_K-\phi_K \|_{L^2(0,T; L^2(\Gamma))} 
	= O(K).
\label{pier13}
\end{gather} 
\end{theorem}

The proof of this theorem is given later. Thanks to \rm (A7), and choosing $m_0'=m_0 \in {\rm int\/} D(\beta_\Gamma)$, we have already obtained uniform estimates in 
Lemmas~\ref{L1}, \ref{L2}, and \ref{L3}, that~is, 
\begin{gather}
	\sqrt{\lambda}  \| \partial_t u_\lambda\|_{L^2(0,T;L^2(\Omega))} 
	+ \| u_\lambda\|_{L^\infty(0,T;H^1(\Omega))} 
	+ \frac{1}{\sqrt{K}} \| u_\lambda-\phi_\lambda \|_{L^\infty(0,T;L^2(\Gamma))} 
	\notag \\
	\quad {}
	+ \| \phi_\lambda \|_{H^1(0,T;L^2(\Gamma))}
	+ \| \phi_\lambda \|_{L^\infty(0,T;H^1(\Gamma))}
	\le M_1,
	\label{est1K}
	\\
	\| \partial _t u_\lambda \|_{L^2(0,T;H^1(\Omega)')} \le M_2,
	\label{est2K}
	\\
	\| \mu_\lambda \|_{L^2(0,T;H^1(\Omega))} 
	\le M_5,
	\label{est5K}
\end{gather}
where $M_1$, $M_2$, and $M_5$ are independent of $\lambda \in (0,1]$ and $K>0$, because 
the terms with factors $1/\sqrt{K}$ and $1/K$ disappear by virtue of the assumptions.  
Moreover, in order to prove Theorem~\ref{K0} we are going to add the following estimates.

\begin{lemma}\label{L}
Under the assumptions {\rm (A1)--(A8)}, there exists a positive constant $M_9$, independent of $\lambda \in (0,1]$ 
and $K>0$, such that
\begin{equation}
	\bigl\| \beta_\lambda (u_\lambda) \bigr\|_{L^2(0,T;L^2(\Omega))}
	+ 
	\bigl\| \beta_\lambda (\phi_\lambda) \bigr\|_{L^2(0,T;L^2(\Gamma))}
	\le M_9.
	\label{est9K}
\end{equation}
\end{lemma}

\begin{proof}
	Multiplying \eqref{dbclam} by $\beta_{\lambda}(\phi_\lambda)$ (please note that it is not $\beta_{\Gamma, \lambda}(\phi_\lambda)$),
	integrating it over $\Gamma$, and using \eqref{rbclam} we deduce for all $\delta >0$ that
\begin{align}
	& 
	\frac{\d}{\dt} \int_\Gamma \hat{\beta}_{\lambda}(\phi_\lambda) \dg
	+
	\int_\Gamma \beta_{\lambda}' (\phi_\lambda) |\nabla_\Gamma \phi_\lambda|^2 \dg  
	+
	\int_\Gamma \beta_{\Gamma,\lambda}(\phi_\lambda) \beta_\lambda(\phi_\lambda) \dg 
	\notag 
	+ 
	\int_\Gamma \partial _{\boldsymbol{\nu}} u_\lambda \beta_{\lambda}(\phi_\lambda)
	\dg 
	\\
	& \le 
	\frac{1}{2\delta} \bigl\| f_\Gamma -\pi_\Gamma (\phi_\lambda) \bigr\|_{L^2(\Gamma)}^2 + 
	\frac{\delta}{2}
	\bigl\| \beta_{\lambda} (\phi_\lambda) \bigr\|_{L^2(\Gamma)}^2. 
	\label{trick2}
\end{align}
Since the signs of $\beta_{\Gamma,\lambda}(r)$ and $\beta_{\lambda}(r)$ coincide, we obtain the following estimate from \eqref{NODEA}~that 
\begin{align}
	\int_\Gamma \beta_{\Gamma,\lambda}(\phi_\lambda)\beta_\lambda(\phi_\lambda)  \dg 
	& = \int_\Gamma \bigl| \beta_{\Gamma,\lambda}(\phi_\lambda) \bigl| \bigl|\beta_\lambda(\phi_\lambda) \bigr| \dg  \notag \\
	& \ge \int_\Gamma \left( \frac{1}{\rho_1} \bigl| \beta_{\lambda}(\phi_\lambda) \bigr|-\frac{c_1}{\rho_1} \right)  \bigl|\beta_\lambda(\phi_\lambda) \bigr| \dg \notag \\
	& \ge \int_\Gamma 
	\left\{ \frac{1}{\rho_1} \bigl| \beta_{\lambda} (\phi_\lambda) \bigr|^2 
	-\left( \frac{1}{2\rho_1} \bigl| \beta_{\lambda} (\phi_\lambda) \bigr|^2
	+\frac{1}{2\rho_1}c_1^2 \right)
	\right\} \dg \notag \\
	& \ge \frac{1}{2\rho_1} \bigl\| \beta_{\lambda} (\phi_\lambda) \bigr\|_{L^2(\Gamma)}^2
	- \frac{c_1^2}{2\rho_1} |\Gamma|. \label{mg1}
\end{align}
Moreover, using the monotonicity of $\beta_{\lambda}$, the boundary condition~\eqref{rbclam} and equation~\eqref{ch2lam} we infer that
\begin{align}
	& \int_\Gamma \partial _{\boldsymbol{\nu}} u_\lambda \beta_{\lambda}(\phi _\lambda)
	\dg 
	= \int_\Gamma \partial _{\boldsymbol{\nu}} u_\lambda 
	\bigl( \beta_{\lambda}(\phi _\lambda)-\beta_{\lambda}(u_\lambda) \bigr)
	\dg 
	+ \int_\Gamma \partial _{\boldsymbol{\nu}} u_\lambda \beta_{\lambda}(u_\lambda)
	\dg \notag \\
	& 
	= 
	\frac{1}{K}
	\int_\Gamma 
		( \phi_\lambda-u_\lambda) 
	\bigl( \beta_{\lambda}(\phi_\lambda)-\beta_{\lambda}(u_\lambda) \bigr)
	\dg 
	+\int_\Gamma \partial _{\boldsymbol{\nu}} u_\lambda \beta_{\lambda}(u_\lambda)
	\dg 
	\ge \int_\Gamma \partial _{\boldsymbol{\nu}} u_\lambda \beta_{\lambda}(u_\lambda)
	\dg \notag \\
	& \ge \int_\Omega \beta_\lambda' (u_\lambda) |\nabla u_\lambda |^2 \dx +
	\bigl(\lambda \partial_t u_\lambda + \pi(u_\lambda) -f -\mu_\lambda, \beta_\lambda (u_\lambda) \bigr)_{L^2(\Omega)}
	+ \bigl\| \beta_\lambda(u_\lambda) \bigr\|_{L^2(\Omega)}^2.
	\label{mg2}
\end{align}
Therefore, using \eqref{mg1} and \eqref{mg2} in \eqref{trick2}, we deduce that
\begin{align*}
	& 
	\frac{\d}{\dt} \int_\Gamma \hat{\beta}_{\lambda}(\phi_\lambda) \dg
	+
	\frac{1}{2\rho_1} \bigl\| \beta_{\lambda} (\phi_\lambda) \bigr\|_{L^2(\Gamma)}^2
	\notag 
	+ 
	\bigl\| \beta_\lambda(u_\lambda) \bigr\|_{L^2(\Omega)}^2 
	\\
	& \le 
	\frac{c_1^2}{2\rho_1} |\Gamma|
	+
	\frac{1}{2\delta} \bigl\| f_\Gamma -\pi_\Gamma (\phi_\lambda) \bigr\|_{L^2(\Gamma)}^2 + 
	\frac{\delta}{2}
	\bigl\| \beta_{\lambda} (\phi_\lambda) \bigr\|_{L^2(\Gamma)}^2
	\notag \\
	& \quad {}
	+
	\frac{1}{2}\bigl\| \lambda \partial_t u_\lambda + \pi(u_\lambda) -f - \mu_\lambda \bigr\|_{L^2(\Omega)}^2 
	+ \frac{1}{2} \bigl\| \beta_\lambda
	(u_\lambda) \bigr\|_{L^2(\Omega)}^2 .
\end{align*}
Thus, taking $\delta:=1/(2\rho_1)$ and integrating the above over $(0,t)$ with respect to the time variable, with the help of \eqref{pier11} we find that
 \begin{align*}
	& 
	\int_\Gamma \hat{\beta}_{\lambda} \bigl(\phi_\lambda (t) \bigr) \dg 
	+
	\frac{1}{4\rho_1} \int_0^t \bigl\| \beta_{\lambda} (\phi_\lambda) \bigr\|_{L^2(\Gamma)}^2 \d \tau
	+ 
	\frac{1}{2} \int_0^t \bigl\| \beta_{\lambda} (u_\lambda) \bigr\|^2 _{L^2(\Omega)} \d \tau
	\notag \\
	& \le 
	\bigl\| \hat{\beta} (\phi_0 ) \bigr\|_{L^1(\Gamma)} +
	\frac{c_1^2}{2 \rho_1}|\Gamma|T +
	\rho_1 \int_0^t \bigl\| f_\Gamma -\pi_\Gamma (\phi_\lambda) \bigr\|_{L^2(\Gamma)}^2 \d\tau 
	\notag \\
	& \quad {} + \frac{1}{2} \int_0^t \bigl\| \lambda \partial_t u_\lambda + \pi(u_\lambda) -f - \mu_\lambda \bigr\|_{L^2(\Omega)}^2  \d \tau.
\end{align*}
Here, from the assumption {\rm (A6)} we see that $\hat{\beta} (\phi_0 ) \in L^1(\Gamma)$. 
Therefore, in view of \eqref{est1K} and \eqref{est5K} it follows that there exists a positive constant $M_9$ such that 
\eqref{est9K} holds. 
\end{proof}

\begin{lemma}\label{Lsim}
Under the assumptions {\rm (A1)--(A8)}, there exist positive constants $M_{10}$, $M_{11}$, $M_{12}$ and $M_{13}$,  independent of $\lambda \in (0,1]$ and 
$K>0$, such that
\begin{gather}
	\| \Delta u_\lambda \|_{L^2(0,T;L^2(\Omega))}
	\le M_{10},
	\label{est10K}\\
	\| \partial_{\boldsymbol{\nu}} u_\lambda \|_{L^2(0,T;H^{-1/2}(\Gamma))}
	\le M_{11}, 
	\label{est11K} 
	\\
	\bigl\| \beta_{\Gamma,\lambda} (\phi_\lambda) \bigr\|_{L^2(0,T;L^2(\Gamma))}
	\le M_{12},  
	\label{est12K}\\
	\| \Delta_\Gamma \phi_\lambda \|_{L^2(0,T;H^{-1/2}(\Gamma))}
	\le M_{13}. 
	\label{est13K}
\end{gather}
\end{lemma}

\begin{proof} 
Firstly, from the comparison of terms in the equation \eqref{ch2lam}, using the uniform estimate 
\eqref{est9K}
of $\{ \beta_\lambda (u_\lambda)\}$ 
in $L^2(0,T;L^2(\Omega))$ obtained in the previous lemma, we can show~\eqref{est10K}. 
Secondly, recalling the generalized Green formula in the first line of~\eqref{Green}, and 
combining \eqref{est1K} and \eqref{est10K} we conclude \eqref{est11K}. 
Thirdly, the estimate in \eqref{NODEA} gives~us 
\begin{equation*}
	\int_0^T \bigl\| \beta _{\Gamma, \lambda} (\phi_\lambda) \bigr\|_{L^2(\Gamma)}^2 \dt 
	\le 2\rho_1 \int_0^T \bigl\| \beta_\lambda (\phi_\lambda) \bigr\|_{L^2(\Gamma)}^2 \dt + 2(c_1 \rho_1)^2 |\Gamma| T.   
\end{equation*}
Therefore, 
we see from the uniform estimate 
\eqref{est9K}
of $\{ \beta_\lambda (\phi_\lambda)\}$ 
in $L^2(0,T;L^2(\Gamma))$ that there exists a positive constant $M_{12}$ such that 
\eqref{est12K} holds. Finally, from the comparison in the equation \eqref{dbclam}
we can also obtain \eqref{est13K}, where the term 
$(1/K) (\phi_\lambda-u_\lambda)$ is replaced by 
\eqref{rbclam} and the uniform estimate \eqref{est11K} is used. 
\end{proof}

\subsection{Proof of Theorem 4.2.}

Before discussing the limiting procedure \( K \to 0 \), we recall the solution constructed in Theorem~\ref{lambda}. Let $ (u_K, \xi_K, \mu_K, \phi_K, \psi_K) $ denote the solution to \eqref{clu}--\eqref{ic2}. Since we are interested in the limit $ K \to 0 $, in contrast to Theorem~\ref{a0} we rewrite the system in terms of the parameter $ K $ instead of $ \alpha $. This reformulation does not alter the dependence of the solution, as the two parameters are related by $ \alpha = 1/K $.
\smallskip

\begin{proof}
From the additional uniform estimates obtained in Lemmas~\ref{L} and \ref{Lsim}, we see that 
$(u_K, \xi_K, \mu_K,\phi_K, \psi_K)$ retains the same estimates as follows: 
\begin{gather}
	\| u_K \|_{L^\infty(0,T;H^1(\Omega))} \le \liminf_{n \to +\infty} \| u_{\lambda_n} \|_{L^\infty(0,T;H^1(\Omega))} \le M_1, \label{es1}
	\\ 
	\frac{1}{\sqrt{K}} \| u_K-\phi_K \|_{L^\infty(0,T;L^2(\Gamma))} 
	\le \liminf_{n \to +\infty} \frac{1}{\sqrt{K}} \| u_{\lambda_n}-\phi_{\lambda_n} \|_{L^\infty(0,T;L^2(\Gamma))} 
	\le M_1, 
	\label{es2} \\
	\| \phi_K \|_{H^1(0,T;L^2(\Gamma))}
	\le 
	\liminf_{n \to +\infty}
	\| \phi_{\lambda_n} \|_{H^1(0,T;L^2(\Gamma))} \le M_1, 
	\label{es3} \\
	\| \phi_K \|_{L^\infty(0,T;H^1(\Gamma))}
	\le \liminf_{n \to +\infty}
	\| \phi_{\lambda_n} \|_{L^\infty(0,T;H^1(\Gamma))}
	\le M_1,
	\label{es4}
	\\
	\| \partial _t u_K \|_{L^2(0,T;H^1(\Omega)')} 
	\le \liminf_{n \to +\infty}
	\| \partial _t u_{\lambda_n} \|_{L^2(0,T;H^1(\Omega)')} \le M_2,
	\label{es5}
	\\
	\| \mu_K \|_{L^2(0,T;H^1(\Omega))} 
	\le \liminf_{n \to +\infty}
	\| \mu_{\lambda_n} \|_{L^2(0,T;H^1(\Omega))} 
	\le M_5,
	\label{es6}
	\\
	\bigl\| \xi_K \bigr\|_{L^2(0,T;L^2(\Omega))}
	\le \liminf_{n \to +\infty}
	\bigl\| \beta_{\lambda_n} (u_{\lambda_n}) \bigr\|_{L^2(0,T;L^2(\Omega))}
	\le M_9, 
	\label{es7}\\
	\| \Delta u_K \|_{L^2(0,T;L^2(\Omega))}
	\le \liminf_{n \to +\infty}
	\| \Delta u_{\lambda_n} \|_{L^2(0,T;L^2(\Omega))}
	\le M_{10},
	\label{es8}\\
	\| \partial_{\boldsymbol{\nu}} u_K \|_{L^2(0,T;H^{-1/2}(\Gamma))}
	\le \liminf_{n \to +\infty}
	\| \partial_{\boldsymbol{\nu}} u_{\lambda_n} \|_{L^2(0,T;H^{-1/2}(\Gamma))}
	\le M_{11},
	\label{es9}\\
	\bigl\| \psi_K \bigr\|_{L^2(0,T;L^2(\Gamma))}
	\le \liminf_{n \to +\infty}
	\bigl\| \beta_{\Gamma,\lambda_n} (\phi_{\lambda_n}) \bigr\|_{L^2(0,T;L^2(\Gamma))}
	\le M_{12},
	\label{es10}\\
	\| \Delta_\Gamma \phi_K \|_{L^2(0,T;H^{-1/2}(\Gamma))}
	\le \liminf_{n \to +\infty}
	\| \Delta_\Gamma \phi_{\lambda_n} \|_{L^2(0,T;H^{-1/2}(\Gamma))}
	\le M_{13}. 
	\label{es11}
\end{gather}
From the properties~\eqref{es1}--\eqref{es11} we see that there exist a subsequence $\{K_n\}$ and a
quintuple $(u, \xi, \mu,\phi, \psi)$ such that 
\begin{align*}
	u_{K_n} \to u & \quad \mbox{\rm weakly~star~in~} 
	H^1 \bigl( 0,T;H^1(\Omega)' \bigr) \cap L^\infty \bigl( 0,T ; H^1(\Omega) \bigr), 
	\\
	\Delta u_{K_n} \to \Delta u & \quad \mbox{\rm weakly~in~} 
	L^2 \bigl( 0,T ; L^2(\Omega) \bigr), 
	\\
	\partial_{\boldsymbol{\nu}} u_{K_n} \to \partial_{\boldsymbol{\nu}} u & 
	\quad \mbox{\rm weakly~in~} 
	L^2 \bigl( 0,T ; H^{-1/2}(\Gamma) \bigr), 
	\\
	\phi_{K_n} \to \phi & \quad 
	\mbox{\rm weakly~star~in~} H^1 \bigl( 0,T;L^2(\Gamma) \bigr) \cap L^\infty \bigl( 0,T ; H^1(\Gamma) \bigr), 
	\\
	\Delta_\Gamma \phi_{K_n} \to \Delta_\Gamma \phi & \quad \mbox{\rm weakly~in~} L^2 \bigl( 0,T;H^{-1/2}(\Gamma) \bigr), 
	\\
	\psi_{K_n} \to \psi & \quad \mbox{\rm weakly~in~} L^2 \bigl( 0,T ; L^2(\Gamma) \bigr),
	\\
	\xi_{K_n} \to \xi & \quad \mbox{\rm weakly~in~}  L^2 \bigl( 0,T ; L^2(\Omega) \bigr), 
	\\ 
	\mu_{K_n} \to \mu & \quad \mbox{\rm weakly~in~}  L^2 \bigl( 0,T ; H^1(\Omega) \bigr)
	\quad \mbox{\rm as~} n \to +\infty.
\end{align*}
Moreover, thanks to the Aubin--Lions compactness theorems (see, e.g., \cite[Sect.~8, Cor.~4]{Sim87}) and 
Lipschitz continuities, we obtain the strong convergences 
\begin{align*}
	u_{K_n} \to u, 
	\quad 
	\pi(u_{K_n}) \to \pi(u) & \quad \mbox{\rm strongly~in~} C\bigl( [0,T]; L^2(\Omega) \bigr), 
	\\
	\phi_{K_n} \to \phi, 
	\quad 
	\pi_\Gamma(\phi_{K_n}) \to \pi_\Gamma(\phi)
	& \quad \mbox{\rm strongly~in~} C\bigl( [0,T]; L^2(\Gamma) \bigr)
	\quad \mbox{\rm as~} n \to +\infty.
\end{align*}
From the demi-closedness properties of $\beta$ and $\beta_\Gamma$ we derive the relations 
\begin{equation*}
	\xi \in \beta(u) \quad {\rm a.e.~in~} Q, \quad 
	\psi \in \beta_\Gamma(\phi) \quad {\rm a.e.~on~}\Sigma. 
\end{equation*}
Additionally \eqref{es2} implies that 
\begin{equation*}
	u=\phi \quad {\rm a.e.~ on ~} \Sigma.
\end{equation*}
In this level, letting $n \to +\infty$ in the weak formulation of \eqref{dbclam} with \eqref{rbclam} 
we actually deduce the following weak formulation of the dynamic boundary condition:
\begin{equation}
\bigl( 
	\partial _t \phi  + \psi + \pi_\Gamma (\phi), \zeta \bigr) _{L^2(\Gamma)}
	+ \langle - \Delta_\Gamma \phi+\partial_{\boldsymbol{\nu}} u, \zeta \rangle_{H^{-1/2}(\Gamma), H^{1/2}(\Gamma)}
	=(f_\Gamma, \zeta)_{L^2(\Gamma)}
	\label{lastw}
\end{equation}
for all $\zeta \in H^{1/2}(\Gamma)$ and for a.a.\ $t \in (0,T)$, on the reason that both
$-\Delta_\Gamma \phi$ and $\partial _{\boldsymbol{\nu}} u$ make sense in $L^2(0,T;H^{-1/2}(\Gamma))$. 
However, by applying linear interpolation theory, we obtain the additional regularity
$
\phi \in L^2(0,T; H^{3/2}(\Gamma)).
$
As a consequence, we may invoke the elliptic estimate (see, e.g., \cite[Theorem 3.2, p.~1.79]{BG87})
\begin{equation*}
	\| u \|_{H^2(\Omega)} \le C_{\rm E} 
	\bigl( \| \Delta  u \|_{L^2(\Omega)} + \| \phi \|_{H^{3/2}(\Gamma)} \bigr)
\end{equation*}
where $C_{\rm E}>0 $ is a constant. Here, we have used the boundary condition $u=\phi$ on $\Sigma$. 
This estimate gives the additional regularities 
\begin{equation*}
	u \in L^2 \bigl( 0,T;H^2(\Omega) \bigr), \quad 
	\partial _{\boldsymbol{\nu}} u \in L^2 \bigl( 0,T; H^{1/2}(\Gamma) \bigr).
\end{equation*}
Moreover, from the subsequent comparison in \eqref{lastw} we infer that
\begin{equation*}
	-\Delta_\Gamma \phi \in  L^2 \bigl( 0,T; L^2(\Gamma) \bigr), \quad 
	\phi \in L^2 \bigl( 0,T; H^2(\Gamma) \bigr).
\end{equation*}
Thus, the proof of the existence part is complete. 
\smallskip

The uniqueness of $ (u, \phi, \psi) $ can be shown by arguments analogous to those used in the proof of Theorem~\ref{lambda}; therefore, we omit the details (see also \cite[Theorem~2.2]{CGS14} or \cite[Theorem~2.1]{CF15pier} for related arguments).
As a consequence of uniqueness, the convergences \eqref{conv1}--\eqref{conv8} for $ u_{K_n} $, $ \phi_{K_n} $, and $ \psi_{K_n} $ hold for the entire sequence.
\smallskip

Finally, we prove the rate of convergence. 
The difference between Theorem~\ref{a0} is the treatment of the boundary condition. 
First, taking care of the fact 
$m(u(t)-u_K(t))=0$ for all $t \in [0,T]$, we deduce 
\begin{equation}
	\frac{1}{2} \frac{\d}{\dt} \| u-u_K \|_{*}^2 
	 +  ( \mu-\mu_K, u-u_K)_{L^2(\Omega)}=0
	\label{rate4}
\end{equation}
a.e.\ on $(0,T)$, based on the choice of the test function 
$y := \mathcal{N} (u-u_K)\in H^1(\Omega)$.
In addition, we have that
\begin{align}
	 & \bigl\| \nabla(u-u_K) \bigr\|_{L^2(\Omega)}^2 
	- \bigl(\partial_{\boldsymbol{\nu}}(u-u_K), u-u_K \bigr)_{L^2(\Gamma)} +
	(\xi-\xi_K, u-u_K)_{L^2(\Omega)} 
	\notag \\
	& \quad {} 
	+ \bigl( \pi(u) -\pi(u_K),u-u_K \bigr)_{L^2(\Omega)} = (\mu-\mu_K,u-u_K)_{L^2(\Omega)}.
	\label{rate5}
\end{align}
Moreover, testing the difference of the dynamic boundary condition of Allen--Cahn type of  
$\phi$ and 
\begin{equation*}
	\partial _t \phi_K - \Delta_\Gamma \phi_K + \psi_K + \pi_\Gamma (\phi_K) + \partial_{\boldsymbol{\nu}} u_K  =f_\Gamma \quad \hbox{a.e.\ on~} \Sigma 
\end{equation*}
(resulting from \eqref{dbc} and \eqref{rbc}), we infer that
\begin{align}
	& \frac{1}{2} \frac{\d}{\dt} \| \phi- \phi_K \|_{L^2(\Gamma)}^2 
	 +  \int_\Gamma \bigl| \nabla _\Gamma (\phi- \phi_K) \bigr|^2 \dg  
	 + (\psi-\psi_K, \phi- \phi_K)_{L^2(\Gamma)}
	 \notag \\
	 & \quad {} + 
	 \bigl( \pi_\Gamma(\phi) -\pi_\Gamma(\phi_K), \phi- \phi_K \bigr)_{L^2(\Gamma)}
	+ \bigl(\partial_{\boldsymbol{\nu}}(u-u_K), \phi- \phi_K \bigr)_{L^2(\Gamma)}=0. \
	\label{rate6}
\end{align}
Take the sum of \eqref{rate4}, \eqref{rate5}, and \eqref{rate6}.  
Using the monotonicity of $\beta$ and $\beta_\Gamma$, the condition $u=\phi$ a.e.\ on $\Sigma$, and again~\eqref{rbc}, we deduce that
\begin{align*}
	& \frac{1}{2} \frac{\d}{\dt} \|u-u_K \|_{*}^2 
	+ \frac{1}{2} \frac{\d}{\dt} \| \phi-\phi_K \|_{L^2(\Gamma)}^2 
	+ \bigl\| \nabla( u-u_K) \bigr\|_{L^2(\Omega)}^2
	\notag \\
	& \quad {} +
	\int_\Gamma \bigl| \nabla _\Gamma (\phi- \phi_K) \bigr|^2 \dg  
	+ \frac{1}{K} \| u_K-\phi_K \|_{L^2(\Gamma)}^2 
		\notag \\
	& \le - \bigl( \pi(u) -\pi(u_K),u-u_K \bigr)_{L^2(\Omega)} 
	- \bigl( \pi_\Gamma(\phi) -\pi_\Gamma(\phi_K), \phi- \phi_K \bigr)_{L^2(\Gamma)}
	\notag \\
	& \ \quad {}
	- \bigl(\partial_{\boldsymbol{\nu}}u, u_K - \phi_K \bigr)_{L^2(\Gamma)}
	\notag \\
	& \le C_{\rm L} \Bigl( \delta \bigl\| \nabla( u-u_K) \bigr\|_{L^2(\Omega)}^2 +  C_\delta \| u-u_K \|_{*}^2 
	+ \| \phi- \phi_K \|_{L^2(\Gamma)}^2 \Bigr)
	\notag \\
	& \ \quad {}	
	+ \frac{K}{2} \| \partial_{\boldsymbol{\nu}} u \|_{L^2(\Gamma)}^2 + \frac{1}{2K} \| u_K-\phi_K \|_{L^2(\Gamma)}^2
\end{align*}
a.e.\ on $(0,T)$. Now, we can fix $\delta:=1/(2C_{\rm L})$, settle also the last term, then integrate with respect to time  and apply the Gronwall lemma to end up with \eqref{takeshi1} and \eqref{pier13}.
\smallskip

Then, Theorem~4.2 is completely proved.
\end{proof}

\section*{Acknowledgments}
This work was supported by the Research Institute for Mathematical Sciences, an International Joint Usage/Research Center located at Kyoto University. 
In addition, P.C.\ acknowledges the support of the Next Generation EU Project No.\ P2022Z7ZAJ (\emph{A unitary mathematical framework for modelling muscular dystrophies}) and of the GNAMPA (Gruppo Nazionale per l'Analisi Matematica, la Probabilit\`a e le loro Applicazioni) of INdAM (Istituto Nazionale di Alta Matematica). 
T.F.\ also acknowledges support from the JSPS KAKENHI Grant-in-Aid for Scientific Research (C), Japan, Grant Number 21K03309.

\end{document}